\documentclass[11pt]{amsart}

\usepackage{amssymb}
\usepackage{enumitem}
\usepackage[dvipsnames]{xcolor}
\usepackage[normalem]{ulem}
\usepackage{hyperref}
\makeatletter
\usepackage{color}

\usepackage[a4paper,top=3.5 cm,bottom=3.5 cm,left=3 cm,right=3 cm]{geometry}

\newtheorem{thm}{Theorem}[section]

\newtheorem{cor}{Corollary}[section]
\newtheorem{lem}{Lemma}[section]
\newtheorem{prop}{Proposition}[section]
\newtheorem{defn}{Definition}[section]
\newtheorem{rem}{Remark}[section]

\def\d{\partial}
\def\ddj{\dot \Delta_j}

\def\hat{\widehat}

\newcommand\C{\mathbb{C}}

\newcommand\Z{\mathbb{Z}}

\newcommand{\N}{\mathbb{N}}

\newcommand{\myref}[1]{}

\renewcommand{\div}{\mbox{\rm div}\;\!}
\def\t{{\langle t\rangle}}
\def\s{{\langle s\rangle}}
\def\X{{\mathcal{X}_{t}}}
\def\Xs{{\mathcal{X}_{s}}}

\def\Xz{{\mathcal{X}_{0}}}

\def\H{{\hat H}}
\def\U{{\hat U}}
\def\HH{{\hat H}_{\varepsilon}}
\def\UU{{\hat U}_{\varepsilon}}

\def\Scho{{Schr$\ddot{o}$dinger\,}}
\def\t{{\langle t\rangle}}
\def\s{{\langle s\rangle}}
\def\XT{{\mathcal{X}_{T}}}
\def\Xt{{\mathcal{X}_{t}}}
\def\Xs{{\mathcal{X}_{s}}}
\def\Xtt{{\mathcal{X}^{2}_{t}}}
\def\Xss{{\mathcal{X}^{2}_{s}}}
\def\Xttt{{\mathcal{X}^{3}_{t}}}
\def\Xsss{{\mathcal{X}^{3}_{s}}}

\def\XTT{{\mathcal{X}^{2}_{T}}}
\def\XTTT{{\mathcal{X}^{3}_{T}}}

\def\Xz{{\mathcal{X}_{0}}}
\def\Xzz{{\mathcal{X}^{2}_{0}}}

\def\E{\mathcal{E}_{T}}

\def\t{{\langle t\rangle}}
\def\s{{\langle s\rangle}}
\def\X{{\mathcal{X}_{t}}}
\def\Xs{{\mathcal{X}_{s}}}

\def\Xz{{\mathcal{X}_{0}}}

\def\R{\mathbb{R}}
\def\B{\mathcal{B}}
\def\C{\mathcal{C}}

\def\H{\hat{H}}

\begin{document}

\title{High Mach number limit for the 3D Euler-Poisson equations of ion dynamics}

\author{Zihao Song}
\address{\it Zihao Song, Mathematics and Key Laboratory of Mathematical MIIT, Nanjing University of Aeronautics and
Astronautics,
Jiangsu Province, Nanjing, 211106, P. R. China\newline
 E-mail: szh1995@nuaa.edu.cn; songzh19950504@gmail.com.}
\date{}



\begin{abstract}
In this paper, we study the global dynamics of the 3D ionic Euler-Poisson equations with the parameter of Mach number $\varepsilon$. We first establish the global well-posedness and scattering for the high Mach number regime  $0<\varepsilon\leq1$ and pressureless case $\varepsilon=0$. Moreover, we prove the high Mach number limit, showing that the profile of the solution for ionic Euler-Poisson equations converged to that of the pressureless equation as $\varepsilon\rightarrow0$. 

Our approach combines energy estimates, dispersive estimates and the normal form method. The major difficulty lies in establishing the uniform estimates with respect to the parameter, as the dispersive or resonance structure degenerates when $\varepsilon$ tends to 0. A crucial observation is that despite the disappearance of the pressure ($\varepsilon\rightarrow0$), dispersive phase function always remains a wave-type structure in zero frequencies,  which enables us to derive linear and bilinear multiplier estimates adapted to the uniformity of Mach number parameter.
\end{abstract}

\maketitle

\section{Introduction}\setcounter{equation}{0}
\hspace*{\fill}

Electrostatic interaction plays a fundamental role in the dynamics of compressible fluids in a wide range of physical settings, including plasma physics, semiconductor modeling, and charged particle flows. In this paper, we investigate the 3D compressible Euler-Poisson system of ion dynamics, in which electronic effect is described as a nonlocal effect of density, and the rigorous derivation of the corresponding equations are given by:
\begin{equation}\label{IEP}
\left\{
\begin{array}{l}\partial_t n^\varepsilon + \nabla \cdot (n^\varepsilon u^\varepsilon) = 0,\\[1mm]
n^\varepsilon (\partial_t u^\varepsilon + u^\varepsilon \cdot \nabla u^\varepsilon) = -\varepsilon\nabla P(n^\varepsilon) - n^\varepsilon  \nabla \phi^\varepsilon \tag{iEP},\\[1mm]
\Delta \phi^\varepsilon =  \bar n e^{\phi^\varepsilon} -n^\varepsilon.\\
 \end{array} \right.
\end{equation}
Here, $n^\varepsilon=n^\varepsilon(t,x)\in \mathbb{R}^{+}$ and $u^\varepsilon=u^\varepsilon(t,x)\in \R^{3}$ are the unknown functions on $[0,+\infty)\times \R^{3}$, which stand for the density
and velocity field of ions, respectively. $\phi^\varepsilon=\phi^\varepsilon(t,x)\in \R^{+}$ represents the self-consistent electric potential, given by the Poisson equation. The parameter $\varepsilon$ is given by $\varepsilon=\frac{1}{\lambda^2}$, where the Mach number $\lambda$ is defined by $\lambda=LT^{-1}\chi^{-1}$ with $L$
and $T$ are the typical values of length and time (before rescaling), $\chi$ stands for the sound speed. In the present paper, we consider the high Mach number situation where $\lambda\geq1$, i.e. $0\leq\varepsilon\leq1$. We assume the pressure satisfies the $\gamma$ law that
\begin{eqnarray}\label{gamma law}
P(n^\varepsilon)=A\left(n^\varepsilon\right)^\gamma,\quad A>0,\gamma>1.
\end{eqnarray}
Without loss of generality, in this paper we always assume $\frac{\gamma A}{\gamma-1}=1$, to simplify the linearization. The initial condition of (iEP) is prescribed by
\begin{equation}\label{1.2}
\left(n^\varepsilon,u^\varepsilon\right)|_{t=0}=\left(n^\varepsilon _{0}(x),u^{\varepsilon}_{0}(x)\right),\  x\in \R^{3}.
\end{equation}
In this paper, we investigate the Cauchy problem of (iEP), where initial data tends to a constant equilibrium $(\bar n,0)$ with $\bar n>0$.

\subsection{Literature review}
The mathematical theory for the Euler–Poisson equations has been thoroughly explored, due to its strong physical background. Fixing $\varepsilon=1$, if one considers the opposite asymptotic limit of the dynamics of a plasma, then \eqref{IEP} reduces to the electronic Euler-Poisson equation, which reads
\begin{equation}\label{EEP}
\left\{
\begin{array}{l}\partial_t \dot n + \nabla \cdot (\dot n u) = 0,\\[1mm]
\dot n (\partial_t u + u \cdot \nabla u) = -\nabla P(\dot n) - \dot n  \nabla \phi,\tag{eEP}\\[1mm]
\Delta \phi = \dot n- \bar n .\\
 \end{array} \right.
\end{equation}

In 1998, Guo \cite{G} first studied \eqref{EEP} in three dimensions and constructed the global smooth solution under small initial amplitude. The motivation in \cite{G} is the observation that the irrotational flow of \eqref{EEP} could be reformulated into a Klein-Gordon type equation near small perturbations. The result was surprising in comparison with works of Sideris \cite{Sideris}, where it was shown that even small and irrotational initial perturbations would blow up in a short time for the pure Euler equation. 

The idea in \cite{G} was later generalized and widely applied in various compressible Euler models. For three dimensional problem, in Guo and Pausader \cite{GP0}, the global scattering solution for \eqref{IEP} when $\varepsilon=1$ was constructed under a normal form argument. For the fluid additionally involves a magnet field, i.e. Euler-Maxwell system, Germain and Masmoudi \cite{GM} established the global well-posedness for one fluid case via so-called space time resonance methods. Subsequently, Guo, Ionescu and Pausader \cite{GIP} addressed the two fluid situation. On the other hand, for the more general non-local force, Choi, Jung and Lee \cite{CJL} considered the Riesz potential and proved the smooth solution in the global sense. For two dimensional cases, global theory is more difficult to prove since the dispersion gets much weaker. For \eqref{EEP}, Jang, Li and Zhang \cite{JLZ}; Li and Wu \cite{LW}; Ionescu and Pausader \cite{IP} proved the Cauchy problem in 2D respectively. Positive results in 2D were given for electron flows under magnet fields was also presented in \cite{DIP}. One could also refer to \cite{AH,DIP,DIPP,GMS1,GMS2,GMS3,GNT1,GNT2,Song} for other related Euler models or dispersive equations in this direction.

There are also fruitful results concerning other aspects of investigation of \eqref{IEP} and \eqref{EEP}.  Threshold for singularity formation of Euler-Poisson equation was first studied in Engelberg, Liu and Tadmor \cite{ELT} and later developed in Liu and Tadmor \cite{LT1,LT2}. For weak solution theory, Chen and Wang \cite{CW} constructed the global solution with specific geometrical structure, by the compensated compactness method. Xiao \cite{Xiao} proved the entropy solution with spherical symmetry. The spherically symmetrical solution with large data were established in recent works of Chen, He, Wang and Yuan \cite{CHWY} and Chen, Huang, Li, Wang and Wang \cite{CHLWW}. Had$\mathrm{\check{z}}$i$\mathrm{\acute{c}}$ and Jang \cite{HJ} established the global solution for both gravitational and plasma case without any symmetrtry assumption. On the other hand, Tadmor and Wei in \cite{TW} studied global regularity of subcritical Euler–Poisson equations. Luo, Rauch,  Xie and Xin \cite{LRXX} investigated the stability of shock solution in 1D. For the quasi-neutral limit, one may refer to Cordier and Grenier \cite{CG}; Wang \cite{W}; Peng and Wang \cite{PW}, as well as series works of Peng and Liu \cite{PL,PL2}.

If one takes $\varepsilon=0$ in \eqref{IEP} or \eqref{EEP}, then they reduce to the pressureless Euler-Poisson equation. In absence of the pressure, the equations are usually used to describe charging transport, cold ions, as well as collapse of stars. For the pressureless case, Nguyen and Tudorascu \cite{NT},  Cavalletti, Sedjro and Westdickenberg \cite{CSW} proved the existence theory in 1D. Carrillo, Choi and  Zatorska \cite{CCZ} studied critical thresholds and large-time behavior with the damping term.
More recently, Choi, Kim, Koo and Tadmor \cite{CKKT} gave the local well-posedness and non-existence results under the certain background state.

As far as we know, existing results about smooth irrotational solution mostly required the existence of pressure law, i.e. $\varepsilon\neq0$, to ensure the non-degenerated monotonicity of the pressure on the constant equilibrium, which plays an essential role in offering the nice dispersive structure. 
In this paper, we aim at investigating the scattering theory of small smooth solution of \eqref{IEP} in high Mach number case ($0<\varepsilon\leq1$) and pressureless case ($\varepsilon=0$), as well as the singularity  limit from \eqref{IEP} towards pressureless \eqref{IEP} when $\varepsilon\rightarrow0$.
The main novelty lies in the uniform analysis with respect to the parameter, which based on the key observation that for ionic flow, {\bf the linear phase always exhibits a wave behavior in zero frequencies, despite the vanishing of the pressure term}, i.e. $\varepsilon\rightarrow0$. Such fact enables us to  estimate those highly singular bilinear multipliers without dependence of parameter $\varepsilon$.

\subsection{Main results}

Without loss of generality, we shall assume Mach number parameter $\varepsilon\in[0,1]$ throughout the paper. We denote by $L^{p}=L^{p}(\mathbb{R}^{3})$ the usual Lebesgue spaces on $\mathbb{R}^{3}$ with the norm $\|\cdot\|_{L^{p}}$. We also denote Bessel potential space $H^k$ and Riesz potential space $\dot{H}^k$ with $s\in\mathbb{R}$, where for $s$ a positive integer they coincide with the usual Sobolev spaces, i.e. $W^{k,p}=H^k$, $\dot{W}^{k,p}=\dot{H}^k$.  It will also be understood that $\|(f,g)\|_{X}=\|f\|_{X}+\|g\|_{X}$ for all $f,g\in X$.
Denote $H_\varepsilon,U_\varepsilon$ by the operator equipped with Fourier symbols $\HH(|\xi|),\UU(|\xi|)$ where the radial functions $\HH(r),\UU(r)$ is given by
\begin{eqnarray}\label{HU}
\HH(r)=\sqrt{|\xi|^2(\frac{1}{1+|\xi|^2}+\varepsilon)},\quad \UU(r)=\sqrt{\frac{1}{1+|\xi|^2}+\varepsilon}.
\end{eqnarray}
We state the first result concerning global well-posedness and scattering for the pressureless (iEP) system:
\begin{thm}\label{thm1}
Let $\varepsilon=0$. Let $s>0$ be sufficiently large, $p=\frac{6}{1-2\delta}$. Let $(n^{0}_{0}-\bar n,u^{0}_{0})\in \dot H^{-1}\cap H^{s}\cap W^{10,p'}$ and 
\begin{eqnarray}\label{initial}
\mathcal{X}^0_0\triangleq\|(n^{0}_{0}-\bar n,U^{-1}_0 u^{0}_{0})\|_{\dot H^{-1}\cap H^s\cap W^{10,p'}}\ll1,
\end{eqnarray}
then the Cauchy problem of (iEP) with initial data $(n^{0}_{0}-\bar n,U^{-1}_0 u^{0}_{0})$ admits a unique global solution $(n^{0}-\bar n,u^{0})$ satisfying $(n^{0}-\bar n,U^{-1}_0 u^{0})\in \mathcal{C}(\mathbb{R}^{+};H^{s})$ while for all $T>0$
\begin{eqnarray}\label{control}
\sup_{t\in[0,T]}\big(\|(n^{0}-\bar n,U^{-1}_0 u^0)\|_{\dot H^{-1}\cap H^s}+\t^{1+\delta}\|(n^{0}-\bar n,U^{-1}_0 u^0)\|_{L^\infty}\big)\leq C\mathcal{X}^0_0,
\end{eqnarray}
where $C>0$ is a constant. Moreover, the solution scatters in $H^s$.
\end{thm}

Secondly we consider the high Mach number case, i.e. $0<\varepsilon\leq1$:
\begin{thm}\label{thm2}
Let $\varepsilon\in(0,1]$. Let $s>0$ be sufficiently large, $p=\frac{6}{1-2\delta}$, $q=\frac{8}{1-3\delta}$. Let $(n^{\varepsilon}_{0}-\bar n,u^{\varepsilon}_{0})\in \dot H^{-1}\cap H^{s}\cap W^{10,p'}\cap W^{10,q'}$ and there exists a small enough positive $\alpha$ such that
\begin{eqnarray}
\mathcal{X}^\varepsilon_0\triangleq\|(n^{\varepsilon}_{0}-\bar n,U^{-1}_\varepsilon u^{\varepsilon}_{0})\|_{\dot H^{-1}\cap H^s\cap W^{10,p'}}+\varepsilon^\alpha\|\varphi_\varepsilon(n^{\varepsilon}_{0}-\bar n,U^{-1}_\varepsilon u^{\varepsilon}_{0})\|_{W^{10,q'}}\ll1,
\end{eqnarray}
then the Cauchy problem of \eqref{IEP} with initial data $(n^{\varepsilon}_{0}-\bar n,U^{-1}_\varepsilon u^{\varepsilon}_{0})$ admits a unique global solution $(n^{\varepsilon}-\bar n,u^{\varepsilon})$ satisfying $(n^{\varepsilon}-\bar n,U^{-1}_\varepsilon u^{\varepsilon})\in \mathcal{C}(\mathbb{R}^{+};H^{s})$ while for all $T>0$
\begin{eqnarray}\label{control2}
\sup_{t\in[0,T]}\big(\|(n^{\varepsilon}-\bar n,U^{-1}_\varepsilon u^\varepsilon)\|_{\dot H^{-1}\cap H^s}+\t^{1+\delta}\|(n^{\varepsilon}-\bar n,U^{-1}_\varepsilon u^\varepsilon)\|_{L^\infty}\big)\leq C\mathcal{X}^\varepsilon_0,
\end{eqnarray}
where the constant $C>0$ is independent of $\varepsilon$. Moreover, the solution scatters in $H^s$.
\end{thm}

Next, we give the convergence process. We state the following result:
\begin{thm}\label{thm3}
Let $(n^0,u^0)$, $(n^\varepsilon,u^\varepsilon)$ be the solution established in Theorem \ref{thm1} and Theorem \ref{thm2}  with corresponding initial data $(n^0_0,u^0_0)$, $(n^\varepsilon_0,u^\varepsilon_0)$ respectively. If it holds that
\begin{eqnarray}
\|(n^\varepsilon_0-n^0_0,u^\varepsilon_0-u^0_0)\|_{L^2}\Rightarrow0,
\end{eqnarray}
then for all $N<s$, there holds:
\begin{eqnarray}
\|e^{iH_\varepsilon t}(n^\varepsilon-\bar n,U^{-1}_\varepsilon u^\varepsilon)-e^{iH_0t}(n^0-\bar n,U^{-1}_0 u^0)\|_{H^N}\Rightarrow0,\quad\mathrm{for}\quad \varepsilon\rightarrow0.
\end{eqnarray}
\end{thm}

\begin{rem}\label{rem}
Now we give some comments in considerations of main theorems:
\begin{itemize}
\item 
The main ingredient for proofs of Theorem \ref{thm1} and \ref{thm2} bases on a better understanding of phase function $\HH(\xi)$ when $\varepsilon$ is small. Our key finding is that multiplier $H_\varepsilon$ always behaves as wave operator in zero frequencies, despite the choice of parameter $\varepsilon$. Actually there holds
\begin{eqnarray}
|\HH(r)|\sim r, |\HH'(r)|\sim 1, |\HH''(r)|\sim r,\quad \mathrm{for\,\, all} \,\,r\leq1,\varepsilon\in[0,1].
\end{eqnarray}
This fact allows us to expect a class of uniform dispersive or bilinear multiplier estimates independent of $\varepsilon$ in zero frequencies, which have been inspired by successful efforts in models with similar wave behaviors such like pressure case in Guo and Pausader \cite{GP0}, or Gross-Pitaevskii equation in Gustafson, Nakanishi and Tsai \cite{GNT1,GNT2}.

\item
If we take $\varepsilon=1$, then results in Theorem \ref{thm2} immediately comes back to the result in Guo and Pausader \cite{GP} under the monotonicity growth of the pressure $P'(\bar n)>0$. 

Moreover, intuitive observation indicates that
$\mathcal{X}^\varepsilon_0\rightarrow\mathcal{X}^0_0$ when $\varepsilon\rightarrow0$. This is due to the issue is that Mach number parameter $\varepsilon$ changes the location of inflection point for $\HH(\xi)$, which slow down the dispersion in the certain frequency region. Indeed, phase function $\HH(\xi)$ is of a ``Concave-Convex" type if $\varepsilon\neq0$, considered in Guo and Pausader \cite{GP0}; Masmoudi and Nakanishi \cite{MN}, where the  inflection point  appears on the sphere with radius of $\varepsilon^{-\frac 14}$. Such inflection point ``travels" towards infinity while parameter is getting smaller, and finally disappear when $\varepsilon=0$, which means $\HH(\xi)$ is entirely concave in pressureless case and facilitates us to derive the global solution with loose integrability in Theorem \ref{thm1}. 

\item We prove the dispersive limit for profile of the solution in Theorem \ref{thm3}, but without purchasing the sharp convergence rates. It would be interesting to find the accurate convergence rates within Sobolev space and we leave this study to the future work.

\item Results in this paper could also be generalized to other Euler models, for example the compressible Euler system with Helmholtz couplings considered in Blanc, Danchin, Ducomet and Ne${\check{c}}$asov${\acute{a}}$ \cite{BDDN}.
\end{itemize}
    
\end{rem}

\subsection{Sketch of proof}
The basic strategy for proving main theorems is quite classical for quasi-linear dispersive system, which originated from  works of Klainerman and Ponce \cite{KP}, Shatah \cite{Sh}. Such an idea contains the following components:

(1) Energy estimates for propagation of Sobolev norms;

(2) Dispersive estimates of uniform decay;

(3) High Mach number limit ($\varepsilon\rightarrow0$) for the profile of solutions.

For controlling the Sobolev energy, although it is based on a rather standard method, we need to establish some new pseudo-product/commutator estimates concerning the parameter $\varepsilon$, see Lemma \ref{eeee} and \ref{commutator}. Then the following uniform energy inequality is obtained:
$$\partial_{t}E_{s}(t)\lesssim\|\langle \nabla\rangle v\|_{L^{\infty}}\cdot E_{s}(t),$$
where $E_{s}(t)$ represents the energy of solutions in Sobolev space while $v=(a,U_{\varepsilon}^{-1}u)$. Consequently the local well-posedness directly follows and our main task is to gathering enough decay rates  for propagating Sobolev norms to the global sense , which was expected from dispersion of irrotational perturbation system.
Now in terms of the irrotational flow, we could reformulate the (iEP) system near the equilibrium into the following integral equation:
\begin{eqnarray}\label{Duhamel}
z^\varepsilon=e^{-iH_\varepsilon t}z^\varepsilon_{0}+\int^{t}_{0}e^{-iH_\varepsilon(t-s)}\mathcal{N}_\varepsilon(z^\varepsilon)ds.
\end{eqnarray}

In order to obtain sufficient decay rates, we shall first establish a  estimate for the linear solution map $e^{iH_\varepsilon t}$. At this stage, the main ingredient is obtaining the uniformity of the parameter $\varepsilon$. Gathering brief analysis shown in Remark \ref{rem}, we establish the uniform dispersive estimate for the operator $e^{iH_\varepsilon t}$ by a delicate frequency decomposition in terms of $\varepsilon$ and corresponding oscillatory estimates, see detailed calculations in Subsection 2.2.

For the sharp decay of nonlinear flow, as illustrated by the so-called ``Strauss exponent" in \cite{St}, the direct application of Strichartz/dispersive estiamtes in 3D is not able to bound quadratic type nonlinearities showing up in $\mathcal{N}_\varepsilon(z^\varepsilon)$. This promotes us to  study the ``second iteration" in the Fourier space of quadratic integral, which reads
$$\int^{t}_{0}\int_{\mathbb{R}^{3}}e^{i\hat\Omega_\varepsilon(\xi,\eta)s}B(\xi,\eta)\hat{f}^\varepsilon(\xi-\eta)\hat{f}^\varepsilon(\eta)d\eta ds,$$
where profile $f^\varepsilon\triangleq e^{iH_\varepsilon t}z^\varepsilon$ and phase functions $\hat{\Omega}_\varepsilon(\xi,\eta)$ are given by \eqref{Omegai}. The key observation is that we can find a uniform lower bound for $\hat{\Omega}_\varepsilon(\xi,\eta)$ despite the choice of $\varepsilon$, which naturally allows us to apply the normal form method and extend the quadratic interactions to higher order ones. At this stage, we establish some new bilinear multiplier estimates, which not only overcome the strong degeneracies brought by the wave fact, but also be uniform in terms of $\varepsilon$,  see detailed analysis in Subsection 2.3.


For the convergence estimates, we follow the strategy considered by the author \cite{Song} by rewriting the error of profile $\tilde f=f^\varepsilon-f^0$ under Duhamel formula \eqref{Duhamel}. Here, the obstacle in bounding $\|\tilde f\|_{L^2}$ is that norm with derivative loss $\|\tilde f\|_{\dot H^2}$ appears, due to the quasi-linearity of the system. We shall decompose the time integral and transform high order Sobolev norm for error into control of a ``local" energy, which leads to initial error and small factor of $\varepsilon$. 

The rest of this paper unfolds as follows. In Section 2, we present the linearization together with some necessary analysis on dispersive estimates and bilinear multiplier estimates, which are the cornerstone for the proof of our main results. In
Section 3, we prove the global well-posedness and scattering of \eqref{IEP} for the pressureless case, i.e. $\varepsilon=0$ while in Section 4, we present the proof for the situation $\varepsilon\in(0,1]$.  In Section 5, we are devoted to proof the dispersive convergence for the profile of the solution when $\varepsilon\ll1$. In Appendix, we recall some basic analysis tools.

\section{Preliminary}\setcounter{equation}{0}

\subsection{Reformulation of (iEP)}
In this subsection, we present the linearizations for \eqref{IEP} system.  Denote $a^\varepsilon=n^\varepsilon-\bar n$ and $u^\varepsilon=\nabla\psi^\varepsilon$, then (iEP) can be reformulated as
\begin{equation}\label{AB}
\left\{
\begin{array}{l}\partial_{t}a^\varepsilon+\Delta \psi^\varepsilon=-\mathrm{div}(a^\varepsilon\nabla\psi^\varepsilon),\\ [1mm]
\partial_{t}\psi^\varepsilon+\varepsilon a^\varepsilon+(1-\Delta)^{-1}a^\varepsilon=\frac{\div}{-\Delta}\div(\nabla \psi^\varepsilon\otimes\nabla \psi^\varepsilon)-\varepsilon \tilde P(a^\varepsilon)+\tilde\phi(a^\varepsilon),\\[1mm]
 \end{array} \right.
\end{equation}
where 
$$\tilde P(a^\varepsilon)=(a^\varepsilon+\bar n)^{\gamma-1}-a^\varepsilon;\quad \tilde\phi(a^\varepsilon)=- \frac{1}{2}(1 - \Delta)^{-1}\left[(1 - \Delta)^{-1}a^\varepsilon\right]^2 + R(a^\varepsilon)$$
with $R(a^\varepsilon)$ fulfills good properties, see (4.4) in Guo and Pausader \cite{GP}. Writing $z^\varepsilon=z^{\varepsilon}_1+iz^{\varepsilon}_2=a^\varepsilon+iU_{\varepsilon}^{-1}\Lambda\psi^\varepsilon$, then $z^\varepsilon$ satisfies the following dispersive equation
\begin{eqnarray}\label{equation z}
\partial_{t}z^\varepsilon+iH_\varepsilon z^\varepsilon=\mathcal{N}_\varepsilon(z^\varepsilon)= F_\varepsilon(z^\varepsilon)-\varepsilon G_\varepsilon(a^\varepsilon),
\end{eqnarray}
where in light of $a^\varepsilon=\frac{1}{2}(z^\varepsilon+\bar z^\varepsilon), 
\psi^\varepsilon=\frac{i}{2}U_\varepsilon\Lambda^{-1}(\bar z^\varepsilon-z^\varepsilon)$, we denote
\begin{eqnarray}
F_\varepsilon(z^\varepsilon)&=&\frac{i}{4}U_{\varepsilon}^{-1}\Lambda^{-1}{\div}\div(\nabla U_\varepsilon\Lambda^{-1}(\bar z^\varepsilon-z^\varepsilon)\otimes\nabla U_\varepsilon\Lambda^{-1}(\bar z^\varepsilon-z^\varepsilon))\\
\nonumber&-&\frac{i}{4}\mathrm{div}((z^\varepsilon+\bar z^\varepsilon)\nabla U_\varepsilon\Lambda^{-1}(\bar z^\varepsilon-z^\varepsilon))+iU_{\varepsilon}^{-1}\Lambda R(\frac{z^\varepsilon+\bar z^\varepsilon}{2});\end{eqnarray}
\begin{eqnarray}
G_\varepsilon(z^\varepsilon)&=&iU_{\varepsilon}^{-1}\Lambda \tilde P(\frac{z^\varepsilon+\bar z^\varepsilon}{2}).\end{eqnarray}

In the following two subsections, we would present some delicate analysis concerns linear dispersive estimates, or bilinear multiplier estimate, which are cornerstones of our scattering theory and singularity limit problem.

\subsection{Dispersive estimates}

In this subsection, let us introduce dispersive estimates for the semi-group $e^{iH_\varepsilon t}$.
\begin{lem}\label{dispersive}
Assume $f$ be the distribution. Let $j_0$ be the integer such that
$$j_0=\max\left\{j\in\Z,2^{j}\leq r_0\right\}$$
where $r_0=\left(1+\frac{\sqrt{3\varepsilon+4\varepsilon^2}}{\varepsilon}\right)^{\frac{1}{2}}$. We have the following dispersive estimates 

(1). $\varepsilon=0$. There holds
\begin{equation}\label{dispersive 1}
\|\ddj e^{iHt} f\|_{L^\infty}\leq C\left\{
\begin{array}{l}|t|^{-\frac{d}{2}}2^{\frac{d-2}{2}j}\|\ddj f\|_{L^{1}},\qquad 2^j\leq 1;\\[1mm]
|t|^{-\frac{d}{2}}2^{2dj}\|\ddj f\|_{L^{1}},\qquad\,\,\,\, 2^j>1,
 \end{array} \right.
\end{equation}
where constant $C>0$.

(2). $0<\varepsilon\leq1$. There holds
\begin{equation}\label{dispersive 2}
\|\ddj e^{iHt} f\|_{L^\infty}\leq C\left\{
\begin{array}{l}|t|^{-\frac{d}{2}}2^{\frac{d-2}{2}j}\|\ddj f\|_{L^{1}},\qquad\qquad\quad 2^j\leq 1;\\[1mm]
|t|^{-\frac{d}{2}}2^{2dj}\|\ddj f\|_{L^{1}},\qquad\qquad\quad\,\,\, 1< 2^j\leq 2^{j_0-2};\\[1mm]
|t|^{-\frac{d-1}{2}-\frac{1}{3}}2^{(2d-\frac 13)j}\|\ddj f\|_{L^{1}},\qquad\!  2^{j_0-2}\leq 2^j\leq 2^{j_0+2};\\[1mm]
|t|^{-\frac{d}{2}}\varepsilon^{-\frac{d}{2}}\|\ddj f\|_{L^{1}},\qquad\qquad\quad\,\,\,   2^{j_0+2}< 2^j\leq \varepsilon^{-\frac 12};\\[1mm]
|t|^{-\frac{d}{2}}2^{\frac d2j}\varepsilon^{-\frac d4}\|\ddj f\|_{L^{1}},\qquad\quad\,\,\,\,\,    2^j>\varepsilon^{-\frac 12},
 \end{array} \right.
\end{equation}
where constant $C>0$ is independent of $\varepsilon$.
\end{lem}

\begin{proof}
We firstly write the following element calculations:
\begin{eqnarray}
\HH'(r)
=\frac{1+\varepsilon(1+r^2)^2}
{({r}^{2}+1)^{\frac 32}(\varepsilon\,{r}^{2}+\varepsilon+1)^\frac{1}{2}};\quad
\HH''(r)
=\dfrac{r\left(\varepsilon\,{r}^{4}-2\,\varepsilon\,{r}^{2}-3\varepsilon-3\right)}
{({r}^{2}+1)^{\frac 52}(\varepsilon\,{r}^{2}+\varepsilon+1)^\frac{3}{2}};
\end{eqnarray}



\begin{multline}\HH'''(r) =\frac{3\left(\varepsilon r^{4}-2\varepsilon r^{3}-2\varepsilon r-\varepsilon-2r-1\right)\left(\varepsilon r^{4}+2\varepsilon r^{3}+2\varepsilon r-\varepsilon+2r-1\right)}{({r}^{2}+1)^{\frac 72}(\varepsilon\,{r}^{2}+\varepsilon+1)^\frac{5}{2}}.\end{multline}


\underline{Case1. $\varepsilon=0$}.
If $\varepsilon=0$, then $\hat H_0$ is a concave function once $r\neq0$, and
$$\H_0(r)\sim\left\{
\begin{array}{l}r,\quad r\leq1;\\ 
1,\quad r\geq 1,\\
 \end{array} \right.\quad\partial_r\hat H_0(r)\sim\left\{
\begin{array}{l}1,\quad\quad r\leq1;\\ 
r^{-3},\quad r\geq1,\\
 \end{array} \right.\quad\partial^2_r\hat H_0(r)\sim\left\{
\begin{array}{l}r\quad\quad\,\, r\leq1;\\
r^{-4},\quad r\geq1,\\[1mm]
 \end{array} \right.$$
 which indicates the zero frequencies behave as wave, while high frequencies' derivative estimates are close to the fractional \Scho semi-group, which might loss derivatives. Then we obtain dispersive estimates for $\varepsilon=0$ by a classical analysis in \cite{G}.

\underline{Case2. $0<\varepsilon\leq1$}. For convenience of calculations, we only focus on $0<\varepsilon\ll1$. In this situation, the inflection point arises on $\H_\varepsilon(r_0)$, which leads to the worse time decay on the certain frequency region. So let handle this first. Notice that if $2^j\in [2^{j_0-2},2^{j_0+2}]$, there holds $|\xi|\sim\varepsilon^{-\frac 14}$ and element calculations show that
$$\H'_\varepsilon(|\xi|)\sim\varepsilon^{\frac 34};\quad
\H''_\varepsilon(|\xi|)\lesssim\varepsilon;\quad
\H'''_\varepsilon(|\xi|)\sim\varepsilon^{\frac 54}.$$
According to the Young inequality, it is sufficient to bound
$$\|I_j(x)\|_{L^\infty},\quad \mathrm{when}\quad 2^j\sim r_0$$
where
\begin{eqnarray}\label{Bessel}
I_j(x)=2^{jd}\int^{\infty}_{0}e^{it\HH(2^{j}r)}\varphi(r)r^{d-1}(r2^{j}|x|)^{-\frac{n-2}{2}}J_{\frac{n-2}{2}}(r2^{j}|x|)dr
\end{eqnarray}
and we start with $|x|\leq2$ where by denoting $D_{r}=\frac{1}{it\HH'(2^j r)2^j}\frac{d}{dr}$, integral by parts in terms of $D_{r}$  immediately yields
\begin{eqnarray}\label{mmm}
I_{j}(x)&=&2^{jd}\int^{\infty}_{0}D^{k}_{r}\big(e^{it\HH(2^{j}r)}\big)\varphi(r)r^{d-1}(r2^{j}|x|)^{-\frac{d-2}{2}}J_{\frac{d-2}{2}}(r2^{j}|x|)dr\\
\nonumber&=&\frac{2^{jd}}{(it2^j)^k}\sum^{k}_{m=0}\sum^{m}_{l_{m}}\int^{\infty}_{0}e^{it\tilde{H}(2^{j}r)}\prod_{l_{m}}\partial^{l_{m}}_{r}\big(\frac{1}{\HH'(2^j r)}\big)\\
\nonumber&\cdot&\partial^{k-m}_{r}\big(\varphi(r)r^{d-1}(r2^{j}|x|)^{-\frac{d-2}{2}}J_{\frac{d-2}{2}}(r2^{j}|x|)\big)dr,
\end{eqnarray}
where $m=l_{1}+l_{2}+...+l_{m}$. Keep in mind that for any $m\geq0$
$$\frac{d^{m}}{d^{m}_{r}}\big(\frac{1}{\HH'(2^j r)}\big)\leq c2^{3j},\,\,\,\mathrm{for}\,\,\,2^{j}\sim\varepsilon^{-\frac{1}{4}},$$
the vanishing property of the Bessel function at the origin indicates
\begin{eqnarray}\label{mmmmm}
|I_{j}(x)|&\leq&C t^{-k}2^{j(d+2k)}.
\end{eqnarray}
Hence, taking $k=\frac{d}{2}$ yields the third inequality in (\ref{dispersive 2}). For case $|x|\geq2$, we rewrite (\ref{Bessel}) into
 \begin{eqnarray}
I_{j}(x)=2^{jd}\int^{\infty}_{0}e^{it\HH(2^{j}r)-2^j r|x|}\varphi(r)r^{d-1}h(2^j r|x|)dr
\end{eqnarray}
where
$$\mathcal{R}(e^{ir}h(r))=cr^{-\frac{d-2}{2}}J_{\frac{d-2}{2}}(r).$$
At this moment, we start with $|x|$ fulfills $\frac{1}{2}\inf\limits_{r}tH'(2^jr)\leq|x|\leq2\sup\limits_{r} tH'(2^jr)$. Notice that this is where inflection appears, however, in light of the fact $$\bar{H}'''(r)\geq  2^{-5j},\quad\bar{H}(r)=\HH(2^{j}r)-\frac{2^jr|x|}{t},$$
by Van der Corput’s lemma and pointwise estimates for $h$ (see in \cite{GPW}), there holds
\begin{eqnarray*}
|I_{j}(x)|\leq C2^{dj}(2^{3j}|t|2^{-5j})^{-\frac{1}{3}}\big|\frac{d}{dr}(\varphi(r)r^{d-1}h(r2^j|x|))\big|\leq t^{-\frac{d-1}{2}-\frac{1}{3}}2^{(2d-\frac 13)j}.
\end{eqnarray*}
For $|x|\leq\frac{1}{2}\inf\limits_{r}t\HH'(2^jr)$ and $2\sup\limits_{r} t\HH'(2^jr)\leq|x|$, there holds
$$\big|\bar{H}'(r)\big|\geq c2^{-3j},\,\,\,c>0,$$
we could repeat same calculations as (\ref{mmm})-(\ref{mmmmm}) and conclude with the third case in (\ref{dispersive 2}).

Now for the rest cases in \eqref{dispersive 2}, we give:

\underline{(1). $2^j\leq 1$.} The first inequality in \eqref{dispersive 2} follow the similar fashion as \eqref{dispersive 1} once we notice that pointwise behaviors:
$$\H_\varepsilon(r)\sim r,\quad\partial_r\H_\varepsilon(r)\sim 1,\quad\partial^2_r\H_\varepsilon(r)\sim r,\quad\partial^{(\alpha)}_r\H_\varepsilon(r)\lesssim r^{-\alpha},\quad \alpha\geq 2, r\leq1.$$

\underline{(2). $1< 2^j< 2^{j_0-2}$.} For the second inequality in \eqref{dispersive 2}, there holds
$$\H_\varepsilon(r)\sim 1,\quad\H'_\varepsilon(r)\sim r^{-3}, \quad|\HH^{(\alpha)}(r)|\lesssim r^{-2-\alpha},\quad \alpha\geq 2, 1\leq r\lesssim \varepsilon^{-\frac 14}.$$
Hence if $|x|\leq2$ or $|x|\geq2, |x|\leq\frac{1}{2}\inf\limits_{r}t\HH'(2^jr)$ or $|x|\geq2,  2\sup\limits_{r} t\HH'(2^jr)\leq|x|$, we could integrate by parts for phase $\H$ or $\bar H$ and have that 
\begin{eqnarray}
|I_{j}(x)|\leq C t^{-k}2^{j(d+2k)}.
\end{eqnarray}
Hence selecting $k=\frac d2$ implies desired estimate. 
Hence we are left with $\frac{1}{2}\inf\limits_{r}tH'(2^jr)\leq|x|\leq2\sup\limits_{r} tH'(2^jr)$. We claim the following inequality holds:
\begin{eqnarray}\label{H2}\H''_\varepsilon(2^jr)\lesssim -2^{-4j},\quad \mathrm{for}\,\,1< 2^j< 2^{j_0-2}.\end{eqnarray}
With help of definition of $j_0$ and localized function, there must hold
\begin{eqnarray}\label{VVV}1\lesssim|\xi|\leq c_1r_0,\quad c_1\in(0,1).\end{eqnarray}
Observing that there holds
\begin{eqnarray*}|\xi|^{4}\H''_\varepsilon(|\xi|)&\lesssim&\varepsilon\,{|\xi|}^{4}-2\,\varepsilon\,{|\xi|}^{2}-3\varepsilon-3\\&=&\varepsilon(c_1^4-1){r_0}^{4}-2\varepsilon(c_1^2-1)r_0^{2}\lesssim-\varepsilon r_0^{-4}\lesssim-1.\end{eqnarray*}
Hence we reach \eqref{H2} and naturally we could have for  $\frac{1}{2}\inf\limits_{r}tH'(2^jr)\leq|x|\leq2\sup\limits_{r} tH'(2^jr)$ that 
\begin{eqnarray*}
|I_{j}(x)|&\leq&C2^{jd}\big(2^{2j}|t|H''(2^jr)\big)^{-\frac{1}{2}}\big|\frac{d}{dr}(\varphi(r)r^{d-1}h(r2^j|x|))\big|\\&\leq&Ct^{-\frac d2}2^{(2d-2)j}\big(H''(2^jr)\big)^{-\frac{1}{2}}\leq t^{-\frac{d}{2}}2^{2dj},
\end{eqnarray*}
which implies the second inequality in \eqref{dispersive 2}. 

\underline{(3). $2^{j_0+2}< 2^j\leq {\varepsilon^{-\frac 12}}$.} 
Next we consider the third inequality in \eqref{dispersive 2},  at this moment, we have pointwise estimates 
$$\H_\varepsilon(r)\sim 1,\quad\H'_\varepsilon(r)\sim \varepsilon r, \quad|\HH^{(\alpha)}(r)|\lesssim \varepsilon r^{2-\alpha},\quad \alpha\geq 2, \varepsilon^{-\frac 14}\leq r\lesssim \varepsilon^{-\frac 12}.$$
For the lower bound of the second order derivatives,  it is not difficult to verify that
$\HH'''(r)$ has two positive zero point, i.e.
$$\HH'''(r_i)=0,\quad\mathrm{where}\,\,\,\,r_1\sim\frac 12, r_2\sim \varepsilon^{-\frac 13}$$
provided $\varepsilon\ll1$. Hence there exists a small $d>0$ such that
$\H''_\varepsilon(2^jr)$ is increasing on $2^{j_0+2}<2^j<d{\varepsilon^{-\frac 13}}$ which ensures
\begin{eqnarray}|\xi|\geq c_2r_0,\quad c_2>1,\quad\Rightarrow\end{eqnarray}
$$\H''_\varepsilon(2^jr)\geq\H''_\varepsilon(c_2r_0)=\dfrac{c_2r_0\left(\varepsilon(c_2^4-1){r_0}^{4}-2\varepsilon(c_2^2-1)r_0^{2}\right)}
{(c_2^2r_0^{2}+1)^{\frac 52}(\varepsilon c_2^2{r}_0^{2}+\varepsilon+1)^\frac{3}{2}}\gtrsim\varepsilon.$$
On the other hand, if $d{\varepsilon^{-\frac 13}}\leq 2^j<{\varepsilon^{-\frac 12}}$, it is obvious that
$$\H''_\varepsilon(2^jr)\gtrsim \varepsilon,$$
which finally indicates
\begin{eqnarray}\H''_\varepsilon(2^jr)\gtrsim \varepsilon,\quad \mathrm{for}\,\,j_0+2< 2^j\leq {\varepsilon^{-\frac 12}}.\end{eqnarray}
Following clues of above proof, we have
for \eqref{mmm} such that
\begin{itemize}
\item $|x|\leq2$
$$|I_j(x)|\leq C t^{-k}2^{(d-2k)j}\varepsilon^{-k};$$

\item $|x|\geq2,\frac{1}{2}\inf\limits_{r}t\HH'(r)\leq|x|\leq2\sup\limits_{r} t\HH'(r)$
$$|I_j(x)|\leq C2^{dj}(2^{2j}\varepsilon t)^{-\frac{1}{2}}(2^{2j} \varepsilon t)^{-\frac{d-1}{2}}\leq C t^{-\frac{d}{2}}\varepsilon^{-\frac d2 j};$$
    
\item $|x|\geq2, |x|\leq\frac{1}{2}\inf\limits_{r}t\HH'(r)$ or $2\sup\limits_{r} t\HH'(r)\leq|x|$
$$|I_j(x)|\leq C t^{-k}2^{(d-2k)j}\varepsilon^{-k}.$$    
\end{itemize}
Hence taking $k=\frac d2$ leads to the fourth inequality in \eqref{dispersive 2}

\underline{(4). $2^j>  \varepsilon^{-\frac 12}$.} 
Finally, in the case $|\xi|\gtrsim\varepsilon^{-\frac 12}$, it is noticed that
$$\H_\varepsilon(r)\sim \varepsilon^{\frac{1}{2}}r,\quad\H'_\varepsilon(r)\sim \varepsilon^{\frac{1}{2}},\quad\H''_\varepsilon(r)\sim \varepsilon^{\frac{1}{2}}r^{-1}, \quad|\HH^{(\alpha)}(r)|\lesssim \varepsilon r^{1-\alpha},\quad \alpha\geq 3, \varepsilon^{-\frac 12}\leq r,$$
which gives rise to
\begin{itemize}
\item $|x|\leq2$
$$|I_j(x)|\leq C t^{-k}2^{(d-k)j}\varepsilon^{-\frac k2};$$

\item $|x|\geq2,\frac{1}{2}\inf\limits_{r}t\HH'(r)\leq|x|\leq2\sup\limits_{r} t\HH'(r)$
$$|I_j(x)|\leq C2^{dj}(\varepsilon^{\frac 12}2^{j}t)^{-\frac{1}{2}}(2^j \varepsilon^{\frac{1}{2}}t)^{-\frac{d-1}{2}}\leq C t^{-\frac{d}{2}}2^{\frac d2j}\varepsilon^{-\frac d4};$$
    
\item $|x|\geq2, |x|\leq\frac{1}{2}\inf\limits_{r}t\HH'(r)$ or $2\sup\limits_{r} t\HH'(r)\leq|x|$
$$|I_j(x)|\leq C t^{-k}2^{(d-k)j}\varepsilon^{-\frac k2}.$$    
\end{itemize}
Therefore taking $k=\frac d2$ concludes the last inequality in \eqref{dispersive 2} and this is Lemma \ref{dispersive}.
\end{proof}

Based on Lemma \ref{dispersive}, we immediately have the following Corollary whose dispersive estimates in different frequency zones, which are independent of parameter $\varepsilon$:
\begin{cor}\label{dispersive0}
Assume $f$ be the distribution and let $p\in[1,\infty]$. Define the multiplier $\varphi_\varepsilon$ under the Fourier symbol
$$\hat\varphi_\varepsilon(\xi)=\varphi\left(\frac{|\xi|-r_{0}}{c_0\varepsilon^{-\frac{1}{4}}}\right),\quad r_0=\left(1+\frac{\sqrt{3\varepsilon+4\varepsilon^2}}{\varepsilon}\right)^{\frac{1}{2}},$$
where $c_0$ is the selected constant. We have the following uniform dispersive estimates:

(1). $\varepsilon=0$
\begin{equation}
\| e^{iH_0t} f\|_{L^p}\leq Ct^{-\frac{d}{2}(1-\frac{2}{p})}2^{\frac{d-2}{2}(\frac{1}{2}-\frac{1}{p})j}\langle2^{\frac{3d+2}{2}(\frac{1}{2}-\frac{1}{p})j}\rangle\|f\|_{L^{p'}},
\end{equation}
where $C$ is positive constant.

(2). $0<\varepsilon\leq1$
\begin{equation}
\|(1-\varphi_\varepsilon) e^{iH_\varepsilon t} f\|_{L^p}\leq Ct^{-\frac{d}{2}(1-\frac{2}{p})}2^{\frac{d-2}{2}(\frac{1}{2}-\frac{1}{p})j}\langle2^{\frac{3d+2}{2}(\frac{1}{2}-\frac{1}{p})j}\rangle\|(1-\varphi_\varepsilon)f\|_{L^{p'}};
\end{equation}
\begin{equation}
\|\varphi_\varepsilon e^{iH_\varepsilon t} f\|_{L^p}\leq C t^{(-\frac{d-1}{2}-\frac{1}{3})(1-\frac{2}{p})}2^{(2d-\frac 13)(\frac{1}{2}-\frac{1}{p})j}\|\varphi_\varepsilon f\|_{L^{p'}},
\end{equation}
where $C$ is positive constant independent of $\varepsilon$.
\end{cor}
The proof of above Corollary is directly obtained by Lemma \ref{dispersive}, once we apply the interpolation and the fact multiplier $\varphi_\varepsilon$ is bounded in $L^p$ space. 

\subsection{Bilinear Multiplier theory}
Bilinear multiplier estimates play crucial rules in our nonlinear analysis, especially dangerous degeneracies appear in those bilinear operator throughout the normal form process. In this subsection we aim at establishing some uniform bilinear multiplier estimates in terms of parameter $\varepsilon$. To begin with, we associate the bilinear “pseudo-product” operator
$$\B[f, g] = \mathcal{F}_{\xi}^{-1} \int_{\mathbb{R}^3} B(\xi, \eta) \hat{f}(\xi - \eta) \hat{g}(\eta) d\eta.$$
The following classical multiplier estimates  were introduced in \cite{GNT1,GNT2}
and generalized by \cite{GP}:
\begin{lem}\label{Multiplier estimates}
Suppose that $0\leq s\leq n/2$ and define
$$\|B(\xi,\eta)\|_{\mathcal{M}^{s}}=\min\left\{\|B(\xi,\eta)\|_{\mathcal{M}^{s}_{\eta,\xi}},\|B(\xi,\eta)\|_{\mathcal{M}^{s}_{\zeta,\xi}}\right\} , \quad \zeta=\xi-\eta$$
where
$$\|B(\xi,\eta)\|_{\mathcal{M}^{s}_{\eta,\xi}}:=\sum_{j\in \mathbb{Z}} \| \dot\Delta_{j}^\eta B(\xi, \eta) \|_{L_\xi^\infty \dot{H}_\eta^s},$$
then
\begin{equation}\label{GNS}
\|\B[f,g]\|_{L^{2}}\lesssim \|B\|_{\mathcal{M}^{s}}\|f\|_{L^{r_{1}}}\|g\|_{L^{r_2}},
\tag{6.3}
\end{equation}
for $r_{1},r_{2}$ satisfying
\[
2\leq r_{1},r_{2}\leq\frac{2n}{n-2s}\quad and\quad\frac{1}{r_{1}}+\frac{1}{r_{2}}=1-\frac{s}{n}.
\tag{6.4}
\]
Moreover, by the duality, there holds
\begin{equation}\label{multiplier}
\|\B[f,g]\|_{L^{l'_{1}}}\lesssim \|B\|_{\mathcal{M}^{s}}\|f\|_{L^{l_{2}}}\|g\|_{L^{2}},
\tag{6.3}
\end{equation}
for $l_{1},l_{2}$ satisfying
\[
2\leq l_{1},l_{2}\leq\frac{2n}{n-2s}\quad and\quad\frac{1}{l_{1}}+\frac{1}{l_{2}}=1-\frac{s}{n}.
\tag{6.4}
\]
\end{lem}

Owing to that stronger degeneracies within the normal form process arise for the wave facts of $\HH(r)$, we shall establish the following lemma concerns our key bilinear estimates:
\begin{lem}\label{key bilinear}Assume $\varepsilon\in[0,1]$.
Let $\zeta=\xi-\eta$ and symbol $B(\xi,\eta)$ satisfies
$$|\nabla_\xi^\alpha\nabla_\sigma^\beta B(\xi,\eta)|\lesssim(|\xi|^{\kappa+1}|\eta||\zeta|)|\xi|^{-\alpha}|\sigma|^{-\beta},\quad\sigma=\eta,\zeta,$$
with $\kappa,\alpha,\beta\geq0$, then if we denote symbol $$m_\varepsilon(\xi,\eta) = \frac{B(\xi,\eta)}{\hat{\Omega}_{\varepsilon}(\xi,\eta)}\varphi(\frac{|\xi|}{a})\varphi(\frac{|\eta|}{b})\varphi(\frac{|\zeta|}{c}),$$
where
\begin{eqnarray}\label{Omegai}
\hat\Omega_\varepsilon(\xi,\eta)\triangleq\H_\varepsilon(|\xi|)\pm\H_\varepsilon(|\eta|)\pm\H_\varepsilon(|\zeta|),\end{eqnarray}
the following inequality holds true for $\vartheta\in(0,\frac 54]$
\begin{eqnarray}\label{bilinear multiplier}
\|m_\varepsilon(\xi,\eta)\|_{\mathcal{M}^{\frac 54-\vartheta}} \lesssim \tilde M^{\kappa}\langle M\rangle^4,\quad\tilde M=\max\{b,c\}, M=\max\{a,b,c\}.\end{eqnarray}
\end{lem}

\begin{proof}

Without loss of generality, it is sufficient to consider $\kappa=0$ and we pay attention to the phase 
$$\hat\Omega_\varepsilon(\xi,\eta)=\H_\varepsilon(|\xi|)-\H_\varepsilon(|\eta|)-\H_\varepsilon(|\zeta|).$$
In light of Bony's frequency decomposition, we split the frequency space into four pieces:
\begin{eqnarray}\label{Bony decom}
1. |\xi|\ll|\eta|\sim|\zeta|,\,\,
2. |\eta|\ll|\xi|\sim|\zeta|,\,\,
3. |\zeta|\ll|\eta|\sim|\xi|,\,\,4. |\xi|\sim|\eta|\sim|\zeta|.
\end{eqnarray}

{\underline{ Case 1. $|\xi|\ll|\eta|\sim|\zeta|$.}} Without loss of generality, we only consider $|\eta|\leq|\zeta|$. In light of the monotonicity  of $\H_\varepsilon(r)$, we have
$$|\hat\Omega_\varepsilon(\xi,\eta)|\gtrsim\H_\varepsilon(|\zeta|)\gtrsim\left\{
\begin{array}{l}|\zeta|,\qquad \quad\,|\zeta|\leq 1;\\[1mm]
1,\qquad\quad\,\,\,\, |\zeta|\in [1,\varepsilon^{-\frac{1}{2}}];\\[1mm]
\varepsilon^{\frac{1}{2}}|\zeta|,\qquad|\zeta|\geq \varepsilon^{-\frac{1}{2}},
 \end{array} \right.$$
also, for derivatives, it holds for $k=1,2$
\begin{eqnarray*}
|\nabla^k_\eta\hat\Omega_\varepsilon|&=&\left|\nabla_\eta^{k-1}\Big(\frac{\H'_\varepsilon(|\eta|)}{|\eta|}\eta-\frac{\H'_\varepsilon(|\zeta|)}{|\zeta|}\zeta\Big)\right|\lesssim|\eta|^{1-k}\left\{
\begin{array}{l}1,\qquad\quad\,\,\,\, |\zeta|\leq 1;\\[1mm]
|\zeta|^{-1},\qquad\,|\zeta|\in[1,\varepsilon^{-\frac{1}{2}}];\\[1mm]
\varepsilon^{\frac{1}{2}},\qquad \quad|\zeta|\geq \varepsilon^{-\frac{1}{2}}.
 \end{array} \right.
\end{eqnarray*}
Naturally utilizing $|\eta|\sim|\zeta|$, one could obtain
\begin{eqnarray}
\Big|\nabla_\eta\big(\frac{1}{\hat\Omega_\varepsilon}\big)\Big|=\Big|\frac{\nabla_\eta\hat\Omega_\varepsilon}{\hat\Omega^2_\varepsilon}\Big|\lesssim\left\{
\begin{array}{l}|\zeta|^{-2},\qquad |\zeta|\leq 1;\\[1mm]
|\zeta|^{-1},\qquad|\zeta|\geq1;
 \end{array} \right.
\end{eqnarray}
\begin{eqnarray*}
\Big|\nabla^2_\eta\big(\frac{1}{\hat\Omega_\varepsilon}\big)\Big|\lesssim\Big|\frac{(\nabla_\eta\hat\Omega_\varepsilon)^2}{\hat\Omega^3_\varepsilon}\Big|+\Big|\frac{\nabla^2_\eta\hat\Omega_\varepsilon}{\hat\Omega^2_\varepsilon}\Big|\lesssim\left\{
\begin{array}{l}|\zeta|^{-3},\quad\, |\zeta|\leq 1;\\[1mm]
|\zeta|^{-2},\,\quad|\zeta|\geq1.
 \end{array} \right.
\end{eqnarray*}
Therefore, utilizing $|\eta|\sim|\zeta|$, we obtain for $k=0,1,2$ such that
\begin{eqnarray}\label{oiuyy}
\Big|\nabla^k_\eta\big(\frac{1}{\hat\Omega_\varepsilon}\big)\Big|\lesssim|\eta|^{-k}\left\{
\begin{array}{l}M^{-1},\qquad M\leq 1;\\[1mm]
1,\qquad\quad\,\,\, M\geq 1.
 \end{array} \right.
\end{eqnarray}
Consequently, elementary calculations enable us to obtain for $s=1,2$ 
\begin{itemize}
    \item $M\leq1$:
\begin{eqnarray}
\|m_\varepsilon(\xi,\eta)\|_{\mathcal{M}^{s}} \lesssim abc\cdot b^{\frac 32-s}M^{-1}\leq C,\end{eqnarray}    
    \item $M\geq1$:
\begin{eqnarray}
\|m_\varepsilon(\xi,\eta)\|_{\mathcal{M}^{s}} \lesssim abc\cdot b^{\frac 32-s}\leq CM^{\frac 92-s},\end{eqnarray}   
\end{itemize}
hence interpolation immediately leads to \eqref{bilinear multiplier}.

{\underline{ Case 2. $|\eta|\ll|\xi|\sim|\zeta|$.}}  We write
\begin{multline}\label{mkmkl}\hat\Omega_\varepsilon(\xi,\eta)=\left[
(|\xi|-|\eta|-|\zeta|)\U_\varepsilon(|\xi|)\right]+\left[|\eta|
(\U_\varepsilon(|\xi|)-\U_\varepsilon(|\eta|))\right]+\left[|\zeta|
(\U_\varepsilon(|\xi|)-\U_\varepsilon(|\zeta|))\right].\end{multline}
Let us firstly give basic calculations on $|\eta|
(\U_\varepsilon(|\xi|)-\U_\varepsilon(|\eta|))$. Actually, we can write
$$\U_\varepsilon(|\xi|)-\U_\varepsilon(|\eta|)=\int^{|\xi|}_{|\eta|}\U'_\varepsilon(s)ds,$$
where $\U'_\varepsilon(r)$ is negative and there holds
\begin{eqnarray}\label{U-1} |\U'_\varepsilon(r)|=\left|\frac{r}{(1+r^2)^2\sqrt{\frac{1}{1+r^2}+\varepsilon}}\right|\sim\left\{
\begin{array}{l}r,\qquad\qquad\,\,\,\, r\leq 1;\\[1mm]
r^{-2},\qquad\quad\,\,\,\, r\in[1,\varepsilon^{-\frac{1}{2}}];\\[1mm]
\varepsilon^{-\frac 12}r^{-3},\qquad r\geq\varepsilon^{-\frac{1}{2}}.
 \end{array} \right.\end{eqnarray} 
Above facts indicates that three parts on the right hand of \eqref{mkmkl} are all negative.
Now we define $\theta\in[0,\pi]$ be the angle of vector $\eta$ and $\xi$, and let us further decompose the Case 2 into

{\underline{ Case 2.1. $|\xi|\lesssim1$.}}
We consider
\begin{itemize}
    \item  $\theta\in[0,\frac{\pi}{2}]$, in this situation, we firstly decompose
$$m_\varepsilon(\xi,\eta) = \chi(\frac{\sin\theta}{c_0|\xi|})m_\varepsilon(\xi,\eta) +\sum_{q\in[0,\log |\xi|^{-1}]}\varphi(\frac{\sin\theta}{c_02^q|\xi|}) m_\varepsilon(\xi,\eta)$$
where constant $c_0$ will be fixed later. Let us begin with the first symbol, in which vector $\eta$ and $\xi$ are ``parallel". Then 
$\U'_\varepsilon(s)\sim s,$ which implies
\begin{eqnarray*}
\Big|(\U_\varepsilon(|\xi|-\U_\varepsilon(|\eta|))\Big|\sim\Big|\int^{|\xi|}_{|\eta|}sds\Big|\sim|\xi|^2
\end{eqnarray*}
\begin{eqnarray*}
\Rightarrow\quad
|\hat\Omega_\varepsilon(\xi,\eta)|\gtrsim\left| |\zeta|
(\U_\varepsilon(|\xi|)-\U_\varepsilon(|\zeta|))\right|\gtrsim|\eta||\xi|^2.
\end{eqnarray*}
Now for derivatives, it holds
\begin{eqnarray*}|\nabla_\eta\hat\Omega_\varepsilon|&=&\Big|\frac{\H'_\varepsilon(|\eta|)}{|\eta|}\eta-\frac{\H'_\varepsilon(|\zeta|)}{|\zeta|}\zeta\Big|\\&\lesssim&|\H'_\varepsilon(|\eta|)(1-\cos \theta)|+|\H'_\varepsilon(|\eta|)-\H'_\varepsilon(|\zeta|)|+|\H'_\varepsilon(|\eta|)\sin \theta|\lesssim|\xi|,\end{eqnarray*}
and similarly
\begin{eqnarray*}|\nabla^2_\eta\hat\Omega_\varepsilon(\xi,\eta)|\lesssim|\xi||\eta|^{-1},\end{eqnarray*}
then \eqref{oiuyy} indicates for $k=0,1,2$ such that
\begin{eqnarray}
\Big|\nabla^k_\eta\big(\frac{1}{\Omega_\varepsilon}\big)\Big|\lesssim|\eta|^{-1}|\xi|^{-2}(|\eta||\xi|)^{-k}.
\end{eqnarray}
For the localized function, basic computations yield there holds
\begin{eqnarray}
|\nabla^k_\eta\chi(\frac{\sin\theta}{c_0|\xi|})|\lesssim(|\eta||\xi|)^{-k},
\end{eqnarray}
we are led to for $s=1,2$ that
\begin{equation}
\Big\|\chi(\frac{\sin\theta}{c_0|\xi|})m_\varepsilon(\xi, \eta)\Big\|_{\mathcal{M}^{s}}\lesssim abc\cdot(a^2b^3)^{\frac 12}b^{-1}a^{-2}(ab)^{-s}\lesssim a^{1-s}b^{\frac 32-s}.
\end{equation}
Consequently, in light of $b\ll a$, we could obtain \eqref{bilinear multiplier} by selecting $\vartheta>0$. On the other hand, if $\theta$ is supported in the ring, there holds
$$|(|\xi|-|\eta|-|\zeta|)\U_\varepsilon(|\xi|)|\geq|\eta|(1-\cos \theta)\U_\varepsilon(|\xi|)\gtrsim|\eta|2^{2q}|\xi|^2$$
$$\Rightarrow\quad|\hat\Omega_\varepsilon(\xi,\eta)|\gtrsim|\eta|2^{2q}|\xi|^2$$
and
\begin{eqnarray*}|\nabla_\eta\hat\Omega_\varepsilon|&\lesssim&2^{2q}|\xi|^2+|\xi|^2+2^q|\xi|\lesssim2^q|\xi|;\quad
|\nabla^2_\eta\hat\Omega_\varepsilon|\lesssim2^q|\xi||\eta|^{-1}.\end{eqnarray*}
Then we arrive at for $k=0,1,2$
\begin{eqnarray}
\Big|\nabla^k_\eta\big(\frac{1}{\hat\Omega_\varepsilon}\big)\Big|\lesssim|\eta|^{-1}2^{-2q}|\xi|^{-2}(|\eta|2^q|\xi|)^{-k}.
\end{eqnarray}
For the localized function, there holds
\begin{eqnarray}
|\nabla^k_\eta\varphi(\frac{\sin\theta}{c_02^q|\xi|})|\lesssim(|\eta|2^q|\xi|)^{-k},
\end{eqnarray}
hence we are led to
\begin{eqnarray*}
\sum_{q\in[0,\log |\xi|^{-1}]}\|\varphi(\frac{\sin\theta}{c_02^q|\xi|})m_\varepsilon(\xi, \eta)\|_{\mathcal{M}^{s}}&\lesssim&\sum_{q\in[0,\log |\xi|^{-1}]} abc\cdot(2^{2q}a^2b^3)^{\frac 12} b^{-1}2^{-2q}a^{-2}(b2^q a)^{-s}\\&\lesssim&
\sum_{q\in[0,\log |\xi|^{-1}]}2^{-(1+s)q}a^{1-s}b^{\frac 32-s}\leq C a^{1-s}b^{\frac 32-s}
\end{eqnarray*}
which completes the situation $\theta\in[0,\frac{\pi}{2}]$.

 \item
$\theta\in[\frac{\pi}{2},\pi]$, at this stage, notice that $|\zeta|\geq|\xi|$, we immediately have
$$|\hat\Omega_\varepsilon(\xi,\eta)|\gtrsim\H_\varepsilon(|\eta|)\gtrsim|\eta|,$$
also, for derivatives, with the facts
$$|\nabla_\eta\hat\Omega_\varepsilon|\lesssim\H'_\varepsilon(|\eta|)\lesssim1;\quad |\nabla^2_\eta\hat\Omega_\varepsilon|\lesssim|\eta|^{-1},$$
then it holds
\begin{eqnarray}
\Big|\nabla^k_\eta\big(\frac{1}{\hat\Omega_\varepsilon}\big)\Big|\lesssim|\eta|^{-1-k},\quad k=0,1,2.
\end{eqnarray}
Hence we are led to
\begin{equation}
\|m_\varepsilon(\xi, \eta)\|_{\mathcal{M}^{s}}\lesssim abc\cdot b^{\frac 32} b^{-1-s}\leq C b^{\frac{3}{2}-s}
\end{equation}
and we finish  the proof of Case 2.1.

\end{itemize}

{\underline{ Case 2.2. $|\xi|\gtrsim 1$, $|\eta|\lesssim d_0\varepsilon^{-\frac 12}$}}
where $d_0\ll1$. Then in light of the monotonicity of $\UU(r)$ and $\HH'(r)$, we have
$$|\hat\Omega_\varepsilon(\xi,\eta)|\gtrsim|\eta|\big|
\U_\varepsilon(|\xi|)-\U_\varepsilon(|\eta|)\big|\gtrsim\left\{
\begin{array}{l}|\eta|,\qquad |\eta|\leq d_0;\\[1mm]
1,\,\,\,\,\qquad d_0\leq|\eta|\leq d_0\varepsilon^{-\frac 12},
 \end{array} \right.$$
 and
$$|\nabla_\eta\hat\Omega_\varepsilon|\lesssim|\H'_\varepsilon(|\eta|)|\lesssim 1;\quad |\nabla^2_\eta\hat\Omega_\varepsilon(\xi,\eta)|\lesssim |\eta|^{-1},$$
which indicate
\begin{eqnarray}
\Big|\nabla^k_\eta\big(\frac{1}{\hat\Omega_\varepsilon}\big)\Big|\lesssim\left\{
\begin{array}{l}|\eta|^{-1-k},\quad |\eta|\leq d_0;\\[1mm]
|\eta|^{-k},\qquad d_0\leq|\eta|\leq d_0\varepsilon^{-\frac 12},
 \end{array} \right.\quad k=0,1,2.
\end{eqnarray}
Hence we have
\begin{itemize}
    \item $|\eta|\leq d_0$:
\begin{equation}
\|m_\varepsilon(\xi, \eta)\|_{\mathcal{M}^{s}}\lesssim abc\cdot b^{\frac 32} b^{-1-s}\leq C a^2b^{\frac{3}{2}-s}.
\end{equation}   
    \item $d_0\leq|\eta|\leq d_0\varepsilon^{-\frac 12}$:
\begin{equation}
\|m_\varepsilon(\xi, \eta)\|_{\mathcal{M}^{s}}\lesssim abc\cdot b^{\frac 32} b^{-s}\leq C a^2b^{\frac{5}{2}-s}.
\end{equation}   
\end{itemize}

{\underline{ Case 2.3. $|\eta|\gtrsim d_0\varepsilon^{-\frac 12}$}.} 
In this case, we have
$$\int^{|\xi|}_{|\eta|}\U'_\varepsilon(s)ds\gtrsim\varepsilon^{-\frac 12}\int^{|\xi|}_{|\eta|}s^{-3}ds\gtrsim\varepsilon^{-\frac 12}|\eta|^{-2} \Rightarrow |\hat\Omega_\varepsilon(\xi,\eta)|\gtrsim\varepsilon^{-\frac 12}|\eta|^{-1}.$$
Noticing that if $m\gtrsim\varepsilon^{-\frac 12}$, then 
$$|\nabla_\eta\hat\Omega_\varepsilon|\lesssim|\H'_\varepsilon(|\eta|)|\lesssim\varepsilon^{\frac 12};\quad |\nabla^2_\eta\hat\Omega_\varepsilon(\xi,\eta)|\lesssim |\eta|^{-1}\varepsilon^{\frac 12}\lesssim 1,$$
we arrive at 
\begin{eqnarray}
\Big|\nabla^k_\eta\big(\frac{1}{\hat\Omega_\varepsilon}\big)\Big|\lesssim\varepsilon^{\frac 12}|\eta|^{-1-k}\Rightarrow\|m_\varepsilon(\xi, \eta)\|_{\mathcal{M}^{s}}\lesssim\varepsilon^{\frac 12}a^2b^{\frac{3}{2}-s}\leq Ca^2b^{\frac{3}{2}-s}.
\end{eqnarray}
This concludes with Case 2. Case 3 exactly follows the symmetrical calculations of Case 2 if we transform the Fourier variable.

{\underline{ Case 4. $|\xi|\sim|\eta|\sim|\zeta|$.}}   For the last case, the support of Fourier variables ensures $\theta\geq\theta_0$ where $\theta_0\in(0,\pi)$, then
\begin{eqnarray}
|\hat\Omega_\varepsilon(\xi,\eta)|\gtrsim
(|\xi|-|\eta|-|\zeta|)\U_\varepsilon(|\xi|)\gtrsim|\eta|\U_\varepsilon(|\xi|)\gtrsim|\eta|\left\{
\begin{array}{l}1,\qquad\quad\,\,\, |\xi|\leq 1;\\[1mm]
|\xi|^{-1},\qquad |\xi|\geq 1,
 \end{array} \right.
\end{eqnarray}
derivatives' estimates could follow the similar fashion as before and we conclude pointwise estimates.
Plugging above calculations together and applying the interpolation, we readily arrive at (\ref{bilinear multiplier}).
\end{proof}

\section{Proof of Theorem \ref{thm1}}\setcounter{equation}{0}

In this section, we aim at proving the pressureless case, i.e. $\varepsilon=0$. The proof relies on establishing the uniform estimates for the following norms:
\begin{eqnarray}\label{norm}
\XT&\triangleq&\sup_{t\in[0,T]}\big[\|z^0\|_{\dot H^{-1}\cap H^{s}}+\t^{1+\delta}\|z^0\|_{W^{5,p}}\big].
\end{eqnarray}
where $s,\delta>0$, $p=\frac{6}{1-2\delta}$. Our main bootstrap Proposition is stated as follows:
\begin{prop}\label{priori}
Let $s$ be sufficiently large, $\delta$ be sufficiently small. Then the following inequality for any $T>0$:
\begin{eqnarray}\label{uniform}
\XT\leq C(\Xz+\Xzz+\XTT+\XTTT),
\end{eqnarray}
where the constant $C>0$.
\end{prop}
The subsections 3.1-3.3 are devoted to prove above key proposition. For convenience of expression, in the following calculations of this section, we always denote $(a^0,u^0,z^0,f^0,H_0,U_0,\mathcal{X}^0_0,\mathcal{N}_0)$ by 
$(a,u,z,f,H,U,\mathcal{X}_0,\mathcal{N})$.

\subsection{High order Sovolev estimates}
In this subsection, we establish the energy estimates with high order Sobolev regularity for the solution. We state the following lemma:
\begin{lem}\label{Energy1}
Let $s$ be sufficient large, $\delta$ be sufficient small. Then the following inequality for any $T>0$:
\begin{eqnarray}\label{high Sobolev}
\|z\|^{2}_{\dot H^{s}}\leq C(\Xz+\XTT+\XTTT),
\end{eqnarray}
where the constant $C>0$.
\end{lem}

\begin{proof}
We define Fourier multiplier
$\mathcal{A}^s:=\Lambda^{s}(1-\Delta)^{\frac{1}{2}}$
and the following inequality is obvious:
$$\|(\Lambda^s a,\mathcal{A}^su)\|_{L^{2}}\sim\|(a,U^{-1}\Lambda\psi)\|^{2}_{\dot H^{s}}.$$
Then its sufficient to give the proof of
\begin{multline}\label{Hs}
\|(\Lambda^s a,\mathcal{A}^su)\|^2_{L^{2}}
\leq C\big(\|(a_0,U^{-1}\Lambda\psi_0)\|^{2}_{\dot H^{s}}
+\int^{T}_{0}G(\|a\|_{L^\infty})\|(a,u)\|_{W^{1,\infty}}\|(a,U^{-1}\Lambda\psi)\|^{2}_{H^{s}}dt\big)
\end{multline}
where $G(\|a\|_{L^\infty})\sim 1+\|a\|_{L^\infty}$ and constant $C>0$. 
Therefore our main task is to prove (\ref{Hs}). To do this, we recall the equation is given by
\begin{equation}
\left\{
\begin{array}{l}\partial_{t}a+u\cdot\nabla a+\div u=-a\div u,\\ [1mm]
\partial_{t}u+u\cdot\nabla u +(1-\Delta)^{-1}\nabla a=\nabla R(a).\\[1mm]
 \end{array} \right.
\end{equation}
Now by imposing multipliers $\Lambda^{s}$ and $\mathcal{A}^s$ with $s>0$ on continuity equation and velocity equation respectively, and impose the $L^2$ inner product, we are able to obtain
\begin{eqnarray*}
\partial_{t}\|\Lambda^s a\|^{2}_{L^2}+\int\Lambda^s(u\cdot\nabla a)\Lambda^s a dx+\int\Lambda^s\textrm{div}u\Lambda^s a dx
=-\int\Lambda^s(a\div u)\Lambda^s a dx;
\end{eqnarray*}
\begin{eqnarray*}
\partial_{t}\|\mathcal{A}^su\|^{2}_{L^2}+\int\mathcal{A}^s(1-\Delta)^{-1}\nabla a\cdot\mathcal{A}^su dx
+\int\mathcal{A}^s(u\cdot\nabla u)\cdot\mathcal{A}^su dx=\int\mathcal{A}^s\nabla R(a)\cdot\mathcal{A}^su dx.
\end{eqnarray*}
In light of that
$$\int\mathcal{A}^s(1-\Delta)^{-1}\nabla a\cdot\mathcal{A}^su dx=-
\int\Lambda^s a\Lambda^s\div u dx,$$
we are led to
\begin{eqnarray*}
\partial_{t}\|(\Lambda^sa,\mathcal{A}^su)\|^{2}_{L^2}
&+&\int\Lambda^s(u\cdot\nabla a)\Lambda^s a dx+\int\mathcal{A}^s(u\cdot\nabla u)\cdot\mathcal{A}^su dx\\&-&\int\Lambda^s(a\div u)\Lambda^s a dx+\int\mathcal{A}^s\nabla R(a)\cdot\mathcal{A}^su dx=0.
\end{eqnarray*}
For the transport terms,  there holds
$$\int\Lambda^s(u\cdot\nabla a)\Lambda^sadx=\int[\Lambda^s,u\cdot\nabla]a\Lambda^sadx+\int u\cdot\nabla \Lambda^sa\Lambda^sadx,$$
where for commutator we have
$$\int[\Lambda^s,u\cdot\nabla]a\Lambda^sadx\lesssim\|u\|_{W^{1,\infty}}\|a\|^{2}_{H^s},$$
while we utilize integrating by part to handle the other term and obtain above inequality. By the similar argument, we could also deduce
$$\int\mathcal{A}^s(u\cdot\nabla u)\mathcal{A}^su dx\lesssim\|u\|_{W^{1,\infty}}\|u\|^{2}_{H^{s+1}}.$$
We also have
$$\int\Lambda^s(a\div u)\Lambda^sadx\lesssim\|(a,u)\|_{W^{1,\infty}}(\|a\|_{H^s}+\|u\|_{H^{s+1}})\|u\|_{H^{s+1}},$$
$$\int\mathcal{A}^s\nabla R(a)\mathcal{A}^sudx\lesssim G(\|a\|_{L^\infty})\|a\|_{W^{1,\infty}}\|a\|_{H^s}\|u\|_{H^{s+1}}.$$
Therefore we conclude with
\begin{eqnarray*}
\partial_{t}\|(\Lambda^sa,\mathcal{A}^su)\|^{2}_{L^2}\lesssim
G(\|a\|_{L^\infty})\|(a,u)\|_{W^{1,\infty}}(\|a\|_{H^s}+\|u\|_{H^{s+1}})^2.
\end{eqnarray*}
Now in light of the fact 
$$\|u\|_{H^{s+1}}\lesssim\|\U^{-1}\Lambda\psi\|_{H^{s}},$$
we arrive at (\ref{Hs}) by integrating on time. Then (\ref{high Sobolev}) is obtained in light of the fact $\|(a,u)\|_{W^{1,\infty}}\lesssim \t^{-1-\delta}$.
\end{proof}

\subsection{Negative order Sovolev estimates}
In this subsection, we establish the energy estimates under $\dot H^{-1}$ space, where the following lemma is built:
\begin{lem}\label{Energy2}
Let $s$ be sufficiently large, $\delta$ be sufficiently small. Then the following inequality for any $T>0$:
\begin{eqnarray}\label{negative Sobolev}
\|z\|^{2}_{\dot H^{-1}}\leq C(\Xz+\XTT+\XTTT),
\end{eqnarray}
where the constant $C>0$.
\end{lem}

\begin{proof}
We apply the Duhamel formula under $\dot H^{-1}$ norm where
\begin{eqnarray}
\|z\|_{\dot H^{-1}}\lesssim\|z_{0}\|_{\dot H^{-1}}+\|\int^{t}_{0}e^{-iH_0(t-s)}\mathcal{N}(z)ds\|_{\dot H^{-1}}.
\end{eqnarray}
Inspired by
\begin{eqnarray}\|\mathcal{N}(z)\|_{\dot H^{-1}}\lesssim(1+\|z\|_{L^\infty})\|z\|_{L^\infty}\|z\|_{L^2}\lesssim \s^{-1-\delta}(\Xss+\Xsss),
\end{eqnarray}
we immediately obtain
\begin{eqnarray}\|\int^{t}_{0}e^{-iH(t-s)}\mathcal{N}(z)ds\|_{\dot H^{-1}}\lesssim \int^{t}_{0}\s^{-1-\delta}(\Xss+\Xsss)ds\lesssim \Xtt+\Xttt.
\end{eqnarray}
Then we arrive at Lemma \ref{Energy2}.
\end{proof}

\subsection{Dispersive estimates}
In this subsection, we aim at establishing the dispersive estimate, the following lemma is established:
\begin{lem}\label{Dispersive1}
Let $s$ be sufficiently large, $\delta$ be sufficiently small.  Then the following inequality for any $T>0$:
\begin{eqnarray}\label{dispersive estimate 1}
\t^{1+\delta}\|z\|_{W^{5,p}}\leq C(\Xz+\XTT+\XTTT),
\end{eqnarray}
where the constant $C>0$.
\end{lem}

\begin{proof}
In light of Duhamel formula \eqref{Duhamel}, we could apply Lemma \ref{dispersive} and Sobolev embedding to derive for the linear part
\begin{eqnarray}
\|e^{iHt}z_0\|_{W^{5,p}}\lesssim \t^{-1-\delta}\|z_0\|_{W^{10,p'}\cap H^{10}}.
\end{eqnarray}
Therefore we are left with nonlinear integral. For our analysis, let us rewrite the nonlinear term \eqref{equation z} by the summarizations of bilinear multiplier and higher order terms:
\begin{eqnarray}\label{nonlinear 2}
\mathcal{N}(z)=\sum_{m\in\N^+}\B_{m}[\Lambda^{-1}z^\pm,\Lambda^{-1}z^\pm]+\mathcal{N}_h(z),
\end{eqnarray}
where $z^\pm$ represents $z$ or $\bar z$, and we apply the Taylor extension on the composite function $R(z)$ in above equality so that $\mathcal{N}_h(z)$ only contains nonlinear terms with order higher than quadratic. Moreover, in light of the fact
\begin{eqnarray}\label{U}
|\nabla^\alpha_\xi\UU(|\xi|)|\leq C|\xi|^{-\alpha}, |\nabla^\alpha_\xi\UU^{-1}(|\xi|)|\leq C\langle |\xi|\rangle|\xi|^{-\alpha}, \quad \mathrm{for\,\, all} \,\,\varepsilon\in[0,1]
\end{eqnarray}
where $C>0$ independent of $\varepsilon$, those quadratic multiplier are equipped with symbol satisfying
$$|\partial_\xi^\alpha\partial_\sigma^\beta B_m(\xi,\eta)|\lesssim (|\xi||\eta||\zeta|)|\xi|^{-\alpha}|\sigma|^{-\beta},\quad\sigma=\eta,\zeta.$$

\subsubsection{Higher order estimates}
Let us begin with cubic and higher order terms, in fact, for cubic terms, it is sufficient to deal with
$$U^{-1}\Lambda(1-\Delta)^{-1}\left((1-\Delta)^{-1}a\right)^3,$$
where dispersive estimates as \eqref{dispersive 1} allow us to directly have
 \begin{eqnarray}\label{qqw}
&&\|\int^{t}_{0}e^{iH(t-s)}U^{-1}\Lambda(1-\Delta)^{-1}\left((1-\Delta)^{-1}a\right)^3ds\|_{W^{5,p}}\\
\nonumber&\lesssim&\int^{t}_{0}\langle t-s\rangle^{-1-\delta}\|\left((1-\Delta)^{-1}a\right)^3\|_{W^{10,p'}\cap H^{10}}ds.
\end{eqnarray}
Now for Sobolev norm, we immediately have
 \begin{eqnarray*}
\|\left((1-\Delta)^{-1}a\right)^3\|_{H^{10}}
\lesssim\|z\|^2_{W^{10,\infty}}\|z\|_{H^{10}}\lesssim\s^{-2-\delta}\Xsss
\end{eqnarray*}
where we apply the following inequality:
\begin{eqnarray}\label{decay 10}
\|z\|_{W^{10,\infty}}\lesssim\|z\|^{\frac{2+\delta}{2+2\delta}}_{W^{5,\infty}}\|z\|^{\frac{\delta}{2+2\delta}}_{H^s}\lesssim\s^{-1-\frac{\delta}{2}}\Xs
\end{eqnarray}
if $s$ is sufficiently large. Therefore inspired by the inequality:
$$\int^{t}_0\langle t-s\rangle^{-a_1}\langle s\rangle^{-a_2} ds\lesssim \langle t\rangle^{-a_2},\quad 0\leq a_1\leq a_2, a_2>1,$$
we deduct
 \begin{multline}\label{Sobo estimates}
\int^{t}_{0}\langle t-s\rangle^{-1-\delta}\|\left((1-\Delta)^{-1}a\right)^3\|_{H^{10}}ds
\lesssim\int^{t}_{1}\langle t-s\rangle^{-1-\delta}\s^{-2-\delta}\Xsss ds\lesssim\t^{-1-\delta}\Xttt.
\end{multline}
On the other hand, to control $W^{10,p'}$ for nonlinearities, we have
$$\|\left((1-\Delta)^{-1}a\right)^3\|_{W^{10,p'}}\lesssim\|z\|^3_{W^{10,3p'}}.$$
Inspired by that
$$\|z\|_{W^{10,3p'}}\lesssim
\|z\|^{\frac{1+4\delta}{3+3\delta}}_{H^{s}}\|z\|^{\frac{2-\delta}{3+3\delta}}_{W^{5,p}}
\lesssim \s^{-\frac{2-\delta}{3}}\Xs$$
which indicates
 \begin{eqnarray}\label{CCC}
&&\int^{t}_{0}\langle t-s\rangle^{-1-\delta}\|\left((1-\Delta)^{-1}a\right)^3\|_{W^{10,p'}}ds\\
\nonumber&\lesssim&\int^{t}_{0}\langle t-s\rangle^{-1-\delta}\s^{-2+\delta}\Xsss ds\lesssim\langle t\rangle^{-1-\delta}\Xttt.
\end{eqnarray}
Consequently combining \eqref{Sobo estimates} and \eqref{CCC}, we deduct decay estimates for cubic terms. Similarly we could handle with even higher order terms, with help of the fact $\|z\|_{L^\infty}\leq c$.

\subsubsection{Quadratic estimates}

So we are left with quadratic terms, for convenience, we simply write $\B_m$ as $\B$ and the following equality holds:
$$\int^{t}_{0}e^{iH(t-s)}\B[\Lambda^{-1}z^\pm,\Lambda^{-1}z^\pm]ds=
e^{iHt}\int^{t}_{0}e^{-i\Omega s}\B[\Lambda^{-1}f,\Lambda^{-1}f]ds$$
where $f$ represents $f=e^{iH t}z$ or $\bar f$ and multiplier $\Omega$ is defined with symbols $\hat\Omega(\xi,\eta)$ given in \eqref{Omegai}. Now in light of Lemma \ref{Multiplier estimates}, we can write in Fourier space such that if $m\neq0$, it holds
\begin{eqnarray}\label{normal form 0}
&&\int^{t}_{0}\int_{\R^2}e^{ix\cdot \xi}e^{i\hat\Omega s}B(\xi,\eta)\widehat{\Lambda^{-1}f}(\xi-\eta)\widehat{\Lambda^{-1}f}(\eta)d\eta ds\\
\nonumber&=&i\int^{t}_{0}\int_{\R^2}e^{ix\cdot \xi}\partial_{s}e^{i\hat\Omega s}\frac{B(\xi,\eta)}{\hat\Omega(\xi,\eta)}\widehat{\Lambda^{-1}f}(\xi-\eta)\widehat{\Lambda^{-1}f}
(\eta)d\eta d\xi,
\end{eqnarray}
then integrating by parts on time $s$ implicates
\begin{multline}\label{normal form}
\int^{t}_{0}e^{i\Omega s}\B[\Lambda^{-1}f,\Lambda^{-1}f]ds=-i e^{i\Omega t}\B_{1}[\Lambda^{-1}f,\Lambda^{-1}f]+i\B_{1}[\Lambda^{-1}z_{0},\Lambda^{-1}z_{0}]\\
+\int^{t}_{0} e^{i\Omega s}(\B_{1}[\Lambda^{-1}\partial_{s}f,\Lambda^{-1}f]+\B_{1}[\Lambda^{-1}f,\Lambda^{-1}\partial_{s}f])ds
\end{multline}
where $\B_{1}$ is equipped with symbol $B_{1}(\xi,\eta)= \frac{B(\xi,\eta)}{\hat\Omega(\xi,\eta)}$. 
We start with the quadratic boundary term where
\begin{eqnarray*}
\|e^{iHt}e^{i\Omega t}\B_{1}[\Lambda^{-1}f,\Lambda^{-1}f]\|_{W^{5,p}}
\leq\|\B_{2}[\Lambda^{-1}z,\Lambda^{-1}z]\|_{H^{5}}.
\end{eqnarray*}
where $\B_{2}$ is equipped with the symbol $B_2(\xi,\eta)=|\xi|^{\frac{3}{2}-\frac{3}{p}}B_1(\xi,\eta)$. Now we apply \eqref{GNS} in Lemma \ref{Multiplier estimates} and \ref{key bilinear}, in which $\kappa=\frac 32-\frac p2, r_1=\tilde p, r_2=p, s=\frac 54-\vartheta, \vartheta>0$ with
\begin{eqnarray}\label{delta p}\tilde p=\frac{12}{5}-f_1 (\delta,\vartheta),\quad \mathrm{where} \,\,\,f_1  (\delta,\vartheta)=\frac{48\delta+48\vartheta}{25+20\delta+20\vartheta},\end{eqnarray}
we obtain
\begin{multline}\label{Decay 111}
\|e^{iHt}e^{i\Omega t}\B_{1}[\Lambda^{-1}f,\Lambda^{-1}f]\|_{W^{5,p}}
\leq\|\Lambda^{-1}z\|_{W^{10,\tilde p}}\|\Lambda^{\frac{1}{2}-\frac{3}{p}}z\|_{W^{10,p}}\leq\t^{-1-\delta}\Xtt
\end{multline}
where inspired by calculations in \eqref{decay 10}, we apply 
\begin{eqnarray}\label{VG}\|\Lambda^{-1}z\|_{W^{10,\tilde p}}\lesssim\|\Lambda^{-1}z\|^{\frac{2}{\tilde p}}_{H^{s}}\|\Lambda^{-1}z\|^{1-\frac{2}{\tilde p}}_{W^{5,\infty}}\lesssim\t^{-\frac{\delta}{2}}\Xt;\end{eqnarray}
$$\|\Lambda^{\frac{1}{2}-\frac{3}{p}}z\|_{W^{10,p}}\lesssim\t^{-1-\frac{\delta}{2}}\Xt,\quad \mathrm{provided}\,\,\vartheta,\delta\ll1.$$
On the other hand, for the initial boundary term, there holds
\begin{eqnarray*}
&&\|e^{iH t}\B_{1}[\Lambda^{-1}z_0,\Lambda^{-1}z_0]\|_{W^{5,p}}\\
&\leq&\t^{-1-\delta}\big(\|\B_{1}[\Lambda^{-1}z_0,\Lambda^{-1}z_0]\|_{W^{10,p'}}+\|\B_{2}[\Lambda^{-1}z_0,\Lambda^{-1}z_0]\|_{ H^{5}}\big).
\end{eqnarray*}
For Sobolev norm, we repeat calculations similarly as \eqref{Decay 111} and derive the desired result. For $L^{p'}$ norm, we take $\kappa=0, l_1=p, s=\frac 54-\vartheta, l_2=2$ in Lemma \ref{Multiplier estimates} and \ref{key bilinear}, and obtain
\begin{eqnarray*}
\|\B_{1}[\Lambda^{-1}z_0,\Lambda^{-1}z_0]\|_{W^{10,p'}}
\leq\|z_0\|^2_{\dot H^{-1}\cap H^{15}}\leq\mathcal{X}^2_0,
\end{eqnarray*}
which indicates
\begin{eqnarray}\label{Decay 222}
\|e^{iH t}\B_{1}[\Lambda^{-1}z_0,\Lambda^{-1}z_0]\|_{W^{5,p}}
\leq \t^{-1-\delta}\Xz^{2}.
\end{eqnarray}
At last, we handle with the quadratic time integral, in which we have
\begin{eqnarray*}
\nonumber&&\|e^{iH t}\int^{t}_{0} e^{i\Omega s}\B_{1}[\Lambda^{-1}\partial_{s}f,\Lambda^{-1}f]ds\|_{W^{5,p}}\\
&\leq&\int^{t}_0\langle t-s\rangle^{-1-\delta}\|\B_{1}[\Lambda^{-1}e^{iHs}\partial_sf,\Lambda^{-1}z]\|_{W^{10,p'}\cap H^{10}}ds.
\end{eqnarray*}
For convenience, we just pay attention to bound the $L^{p'}$ norm while 
Sobolev norm's control is much easier (So does in the follwing calculations). We apply Lemma \ref{Multiplier estimates} and \ref{key bilinear}, in which $l_1=\frac{6}{1-2\delta}, s=\frac 54-\vartheta, \vartheta>0, l_2=\tilde p $,
then if $\delta,\vartheta\ll1$, there naturally holds
\begin{eqnarray*}
\nonumber\|\B_{1}[e^{iHs}\partial_s f,z]\|_{W^{10,p'}}\lesssim\|\Lambda^{-1}e^{iHs}\partial_s f\|_{H^{15}}
\|\Lambda^{-1} z\|_{W^{\tilde p,15}}\lesssim\|\Lambda^{-1}e^{iHs}\partial_s f\|_{H^{15}}
\|\Lambda^{-1} z\|_{H^{s}}.
\end{eqnarray*}
In light of the fact
$$\Lambda^{-1}e^{iHs}\partial_s f=\Lambda^{-1}\mathcal{N}(z)$$
and further
\begin{eqnarray}
\|\Lambda^{-1}\mathcal{N}(z)\|_{H^{15}}\lesssim(1+\|z\|_{L^\infty})
\|z\|^2_{W^{20,4}}\lesssim \t^{-1-\delta}\Xss
\end{eqnarray}
in which we use
$$\|z\|_{W^{20,4}}\lesssim\|z\|^{\frac{1+4\delta}{4(1+\delta)}}_{H^{s}}\|z\|^{\frac{3}{4(1+\delta)}}_{W^{5,p}}\lesssim\t^{-\frac{1+\delta}{2}}\Xs,$$
hence we finally arrive at
\begin{multline}\label{Decay 333}
\|e^{iH t}\int^{t}_{0} e^{i\Omega s}\B_{1}[\Lambda^{-1}\partial_{s}f,\Lambda^{-1}f]ds\|_{W^{5,p}}
\leq\int^{t}_0\langle t-s\rangle^{-1-\delta}\s^{-1-\delta}\Xsss ds\leq \t^{-1-\delta}\Xttt.
\end{multline}
The other time integral could be handled in a similar fashion and combining \eqref{Decay 111}, \eqref{Decay 222} and \eqref{Decay 333}, we finally arrive
$$\Big\|\int^{t}_{0}e^{iH(t-s)}\B[\Lambda^{-1}z^\pm,\Lambda^{-1}z^\pm]ds\Big\|_{W^{5,p}}\lesssim\t^{-1-\delta}\big(\mathcal{X}^2_0+\Xtt+\Xttt\big)$$
and this finish the proof of Lemma \ref{dispersive 1}.

\end{proof}

\subsection{Well-posedness and scattering}\setcounter{equation}{0}
Combing Lemma \ref{Energy1}, \ref{Energy2} and \ref{Dispersive1}, we naturally arrive at Proposition \ref{priori}. Consequently, combining the local result, which is obtained by energy estimates (Lemma \ref{Energy1}, \ref{Energy2}) and a standard method in Kato \cite{K}, with a bootstrap argument, we arrive at the global solution and \eqref{control} in Theorem \ref{thm1} once \eqref{initial} is satisfied.  For scattering result, in light of the normal form process in subsection 3.4.1, we have
\begin{multline*}
\lim_{t\rightarrow\infty }\|z(t)-e^{-iHt}z_{0}\|_{L^2}
\lesssim \int^{\infty}_{t}\|e^{iHs}\mathcal{N}(z)\|_{L^2}ds
\lesssim\int^{\infty}_{t}\|z\|_{W^{2,\infty}}\|z\|_{H^2}ds\lesssim t^{-\delta}.
\end{multline*}
In light of the fact
$$\lim_{t\rightarrow\infty }\|z(t)-e^{-iHt}z_{0}\|_{\dot H^s}
\lesssim C,$$
the scattering is obtained by interpolation and this finishes the proof of Theorem \ref{thm1}.

\section{Proof of Theorem \ref{thm2}}\setcounter{equation}{0}

In this section, we show the well-posedness and scattering under $\varepsilon\in(0,1]$. The functional space is selected as follows:
\begin{eqnarray}
\XT&\triangleq&\sup_{t\in[0,T]}\big[\|z^\varepsilon\|_{\dot H^{-1}\cap H^{s}}+\t^{1+\delta}\|(1-\varphi_\varepsilon)z^\varepsilon\|_{W^{5,p}}+\t^{1+\delta}\|\varphi_\varepsilon z^\varepsilon\|_{W^{5,q}}\big].
\end{eqnarray}
where $s,\delta>0$, $p=\frac{6}{1-2\delta}$ and $q=\frac{8}{1-3\delta}$. Our main bootstrap Proposition is stated as follows:
\begin{prop}\label{priori2}
Let $\varepsilon\in(0,1]$. Let $ s$ be sufficiently large and $\delta$ be sufficiently small.  Then the following inequality for any $T>0$:
\begin{eqnarray}\label{uniform2}
\XT\leq C(\Xz+\Xzz+\XTT+\XTTT),
\end{eqnarray}
where the constant $C>0$ independent of $\varepsilon$.
\end{prop}
We follow the strategy as pressureless case, to bound above uniform norm respectively. Again, for convenience of expression, in this section, we always denote $(a^\varepsilon,u^\varepsilon,z^\varepsilon,f^\varepsilon,H_\varepsilon,U_\varepsilon,\mathcal{X}^\varepsilon_0,\mathcal{N}_\varepsilon)$ by 
$(a,u,z,f,H,U,\mathcal{X}_0,\mathcal{N})$.

\subsection{Energy estimates}
In this subsection, we establish the energy estimates for the solution. We state the following lemma concerning the energy estimates:
\begin{lem}\label{Energy3}
Let $s$ be sufficiently large and $\delta$ be sufficiently small. Then the following inequality for any $T>0$:
\begin{eqnarray}\label{high Sobolev1}
\|z\|^{2}_{\dot H^{s}}\leq C(\Xz+\XTT+\XTTT),
\end{eqnarray}
where the constant $C>0$ independent of $\varepsilon$.
\end{lem}

\begin{proof}
Let us start with introducing the following Fourier multipliers:
$$\mathcal{A}_1^s:=\Lambda^{s}U^{-1};\quad\mathcal{A}_2^s:=\Lambda^{s}U^{-1}(1-\Delta)^{-\frac{1}{2}}$$
and the rescaled density:
\begin{equation}\label{eq:c}  
c\triangleq\frac{\varepsilon^{\frac{1}{2}}}{\check\gamma}\sqrt{\frac{\d P}{\d\rho}}=\frac{\varepsilon^{\frac{1}{2}}\gamma^\frac{1}{2}}{\check\gamma}n^{\check\gamma},\quad \check{\gamma}=\dfrac{\gamma-1}{2}.
\end{equation}
Setting  $\tilde c:= c-\bar{c}$ where $\bar c=\frac{(4\varepsilon\gamma)^{\frac{1}{2}}}{\gamma-1}$,
system \eqref{IEP} is rewritten as
\begin{equation} \left\{ \begin{aligned} 
&\partial_t \tilde c+u\cdot\nabla \tilde c+\check{\gamma}(\tilde c+\bar c)\textrm{div}u=0,\\
&\partial_tu+u\cdot\nabla u+\check{\gamma}(\tilde c+\bar{c})\nabla \tilde c=\nabla \phi. \end{aligned} \right.\label{CED4}
\end{equation}
Now frist of all, let us claim the following inequality:
\begin{eqnarray}\label{eqlize}
c_1\|(\mathcal{A}_1^s\tilde c,\mathcal{A}_1^su,\mathcal{A}_2^sa)\|_{L^2}\lesssim\|(a,U^{-1}\Lambda \psi)\|^{2}_{\dot H^{s}}\lesssim c_2\|(\mathcal{A}_1^s\tilde c,\mathcal{A}_1^su,\mathcal{A}_2^sa)\|_{L^2},
\end{eqnarray}
where $c_1\lesssim1\lesssim c_2$. Indeed, for the left side, we only pay attention to density and have
\begin{eqnarray}\label{mmnn}
\|\mathcal{A}_1^s\tilde c\|_{L^2}\lesssim(1+\|a\|_{L^\infty})\|\Lambda^s U^{-1}\varepsilon^{\frac 12}a\|^{2}_{L^{2}}\lesssim \|a\|_{\dot H^{s}},
\end{eqnarray}
where we apply the Plancherel theory and the fact $\varepsilon^{\frac 12}\U^{-1}(|\xi|)\lesssim1$. Similarly, utilizing $\U^{-1}(|\xi|)(1+|\xi|^2)^{-\frac{1}{2}}\lesssim1$, we find
$$\|\mathcal{A}_2^sa\|_{L^2}\lesssim \|a\|_{\dot H^{s}},$$
and this concludes the left side of \eqref{eqlize}. On the other hand, for the right side, the inequality
$$1\leq\U^{-1}(|\xi|)(\varepsilon^{\frac 12}+\frac{1}{\sqrt{1+|\xi|^2}})$$
immediately yields
$$\|a\|_{\dot H^{s}}\lesssim \|\Lambda^sU^{-1}(\varepsilon^{\frac 12}a,(1-\Delta)^{-\frac 12}a)\|_{L^2}
\lesssim \|(\mathcal{A}_1^s\tilde c,\mathcal{A}_2^sa)\|_{L^2}$$
and we arrive at \eqref{eqlize}. Inspired by \eqref{eqlize}, it is sufficient to prove the following inequality:
\begin{multline}\label{Hs11}
\|(\mathcal{A}_1^s\tilde c,\mathcal{A}_1^su,\mathcal{A}_2^sa)\|_{L^2}^{2}
\leq C\big(\|(\mathcal{A}_1^s\tilde c_0,\mathcal{A}_1^su_0,\mathcal{A}_2^sa_0)\|_{L^2}^{2}
\\+\int^{T}_{0}G(\|a\|_{L^\infty})\|(a,u)\|_{W^{1,\infty}}\|(a,U^{-1}\Lambda\psi)\|^{2}_{H^{s}}dt\big)
\end{multline}
where constant $C>0$ independent of $\varepsilon$.

Now by imposing multiplier $\mathcal{A}_1^s$ with $s>0$ under the $L^2$ inner product, we write above equation into
\begin{eqnarray}\label{HN1}
\partial_{t}\|\mathcal{A}_1^s\tilde c\|^{2}_{L^2}+\check{\gamma}\bar c\int\mathcal{A}_1^s\textrm{div}u\mathcal{A}_1^s\tilde c dx
+\int\mathcal{A}_1^sF_{1}\mathcal{A}_1^s\tilde c dx=0;
\end{eqnarray}
\begin{eqnarray}\label{HN2}
\partial_{t}\|\mathcal{A}_1^su\|^{2}_{L^2}+\check{\gamma}\bar c\int\mathcal{A}_1^s\nabla\cdot\tilde c\mathcal{A}_1^s u dx
+\int\mathcal{A}_1^sF_{2}\cdot\mathcal{A}_1^s u dx=\int\mathcal{A}_1^s\nabla\phi\cdot\mathcal{A}_1^s u dx,
\end{eqnarray}
where
$$F_{1}=u\cdot\nabla \tilde c+\check{\gamma}\tilde c\textrm{div}u,\quad F_{2}=u\cdot\nabla u+\check{\gamma}\tilde c\nabla \tilde c.$$
Let us estimate $F_{i}$. For the transport terms, using integral by parts, there holds
\begin{eqnarray}\label{trans}
\int\mathcal{A}_1^s(u\cdot\nabla \tilde c)\mathcal{A}_1^s\tilde c dx&=&\int u\cdot\nabla \mathcal{A}_1^s\tilde c\mathcal{A}_1^s\tilde c dx+\int[\mathcal{A}_1^s,u\cdot\nabla] \tilde c\mathcal{A}_1^s\tilde c dx\\ \nonumber&=&-\frac 12\int\div u\mathcal{A}_1^s \tilde c\mathcal{A}_1^s\tilde c dx+\int[\mathcal{A}_1^s,u\cdot\nabla] \tilde c\mathcal{A}_1^s\tilde c dx.
\end{eqnarray}
There immediately holds
$$\int\div u\mathcal{A}_1^s \tilde c\mathcal{A}_1^s\tilde c dx\lesssim\|\div u\|_{L^{\infty}}\|\mathcal{A}_1^s\tilde c\|^{2}_{L^{2}}\lesssim\|\div u\|_{L^{\infty}}\|a\|^{2}_{H^{s}},$$
where we apply \eqref{mmnn}. On the other hand, for commutator, we apply Lemma \ref{commutator} to obtain
\begin{multline}\int[\mathcal{A}_1^s,u\cdot\nabla] \tilde c\mathcal{A}_1^s\tilde c dx\lesssim\|[\mathcal{A}_1^s,u\cdot\nabla] \tilde c\|_{L^2}\|a\|_{H^s}\lesssim\varepsilon^{-\frac 12}\big(\|\nabla \tilde c\|_{L^\infty}\|u\|_{H^s}\\
+\|\nabla u\|_{L^\infty}\|\tilde c\|_{H^s}\big)\|a\|_{H^s}\lesssim\|\nabla (a,u)\|_{L^\infty}\|(a,U^{-1}\Lambda\psi)\|^2_{H^s},\end{multline}
where we apply $1\leq\U^{-1}(|\xi|)$ and the fact $\varepsilon^{-\frac 12}c\sim a$ in the last inequality. Above inequality yields
\begin{eqnarray}
\int\mathcal{A}_1^s(u\cdot\nabla \tilde c)\mathcal{A}_1^s\tilde c dx\lesssim\|\nabla (a,u)\|_{L^\infty}\|(a,U^{-1}\Lambda\psi)\|^2_{H^s}.
\end{eqnarray}
Similarly we can deal with the other transport term. For rest integral, there holds
$$\int\mathcal{A}_1^s(\tilde c\textrm{div}u)\mathcal{A}_1^s\tilde c dx=\int\tilde c\textrm{div}\mathcal{A}_1^su\mathcal{A}_1^s\tilde c dx+\int [\mathcal{A}_1^s,\tilde c\div]u\mathcal{A}_1^s\tilde c dx;$$
$$\int\mathcal{A}_1^s(\tilde c\nabla \tilde c)\cdot\mathcal{A}_1^s u dx=\int \tilde c\nabla \mathcal{A}_1^s\tilde c\cdot \mathcal{A}_1^su dx+\int [\mathcal{A}_1^s,\tilde c\nabla]\tilde c\cdot \mathcal{A}_1^su dx,$$
again for commutator there holds
\begin{eqnarray*}
&&\int [\mathcal{A}_1^s,\tilde c\div]u\mathcal{A}_1^s\tilde c dx+\int [\mathcal{A}_1^s,\tilde c\nabla]\tilde c\cdot\mathcal{A}_1^s u dx\\
&\lesssim&\varepsilon^{-\frac 12}\big(\|\nabla \tilde c\|_{L^\infty}\|u\|_{H^s}
+\|\nabla u\|_{L^\infty}\|\tilde c\|_{H^s}\big)\|(a,U^{-1}\Lambda\psi)\|_{H^{s}}\\
&\lesssim&\|\nabla (a,u)\|_{L^\infty}\|(a,U^{-1}\Lambda\psi)\|^2_{H^s}.
\end{eqnarray*}
At last we have
\begin{eqnarray*}
\int\tilde c\textrm{div}\mathcal{A}_1^su\mathcal{A}_1^s\tilde c dx+\int \tilde c\nabla \mathcal{A}_1^s\tilde c\cdot\mathcal{A}_1^su dx=-\int\nabla\tilde c\cdot\mathcal{A}_1^su\mathcal{A}_1^s\tilde c dx\lesssim\|\nabla\tilde c\|_{L^\infty}\|(a,U^{-1}\Lambda\psi)\|_{H^s}.
\end{eqnarray*}
Consequently we obtain
$$\int\mathcal{A}_1^sF_{1}\mathcal{A}_1^s\tilde c dx+\int\mathcal{A}_1^sF_{2}\cdot\mathcal{A}_1^su dx\lesssim\|(a,u)\|_{W^{1,\infty}}\|(a,U^{-1}\Lambda\psi)\|^{2}_{H^{s}}.$$
On the other hand, for the term 
$$\int\mathcal{A}_1^s\nabla\phi\cdot\mathcal{A}_1^su dx=-\int\mathcal{A}_1^s(1-\Delta)^{-1}a\mathcal{A}_1^s\div u dx+\int\mathcal{A}_1^s\nabla R(a)\cdot\mathcal{A}_1^su dx.$$
It is inspired by Lemma \ref{eeee} that
$$\int\mathcal{A}_1^s\nabla R(a)\cdot \mathcal{A}_1^su dx\lesssim\|a\|_{W^{1,\infty}}\|a\|_{H^{s}}\|U^{-1}\Lambda\psi\|_{H^{s}}.$$
Hence we are left with bounding $\int\mathcal{A}_1^s(1-\Delta)^{-1}a\mathcal{A}_1^s\div u dx$. We consider the continuity equation, where inner product with $\mathcal{A}_2^s$ yields
\begin{multline}\label{HN3}
\partial_{t}\|\mathcal{A}_2^s a\|^{2}_{L^2}+\int\mathcal{A}_2^s\textrm{div}u\mathcal{A}_2^s a dx
+\int\mathcal{A}_2^s(a\textrm{div}u)\mathcal{A}_2^s a dx
+\int\mathcal{A}_2^s(u\cdot\nabla a)\mathcal{A}_2^s a dx=0.
\end{multline}
There holds
$$\int\mathcal{A}_2^s\textrm{div}u\mathcal{A}_2^s a dx=\int\mathcal{A}_1^s\textrm{div}u\mathcal{A}_1^s(1-\Delta)^{-1} a dx$$
Applying Lemma \ref{eeee} which indicates
\begin{eqnarray*}
\int\mathcal{A}^s_2(a\textrm{div}u)\mathcal{A}^s_2 a dx\lesssim\|(a,u)\|_{W^{1,\infty}}\|(a,\Lambda u)\|^{2}_{H^{s}}\lesssim
\|(a,u)\|_{W^{1,\infty}}\|(a,U^{-1}\Lambda\psi)\|^{2}_{H^{s}},
\end{eqnarray*}
\begin{multline*}
\int\mathcal{A}_2^s(u\cdot\nabla a)\mathcal{A}^s_2 a dx=
\int[\mathcal{A}^s_2,u\cdot\nabla] a\mathcal{A}^s_2 a dx
+\int u\cdot\nabla\mathcal{A}^s_2 a\mathcal{A}^s_2 a dx\\ \lesssim\|(a,u)\|_{W^{1,\infty}}\|(a,U^{-1}\Lambda\psi)\|^{2}_{H^{s}}.
\end{multline*}
Therefore, plugging \eqref{HN1}-\eqref{HN2} with \eqref{HN3}, we have
\begin{eqnarray}
\partial_{t}\|(\mathcal{A}_1^s\tilde c,\mathcal{A}_1^su,\mathcal{A}_2^sa)\|^{2}_{L^2}\lesssim G(\|a\|_{L^\infty})\|(a,u)\|_{W^{1,\infty}}\|(a,U^{-1}\Lambda\psi)\|^{2}_{H^{s}},
\end{eqnarray}
then integrating on time leads to (\ref{Hs11}) and this concludes the proof of Lemma \ref{Energy3}.
\end{proof}

\subsection{Negative order Sovolev estimates}
In this subsection, we establish the energy estimates under $\dot H^{-1}$ space, where the following lemma would be built:
\begin{lem}\label{Energy4}
Let $s$ be sufficiently large and $\delta$ be sufficiently small. Then the following inequality for any $T>0$:
\begin{eqnarray}\label{negative Sobolev 2}
\|z\|^{2}_{\dot H^{-1}}\leq C(\Xz+\XTT+\XTTT),
\end{eqnarray}
where the constant $C>0$ independent of $\varepsilon$.
\end{lem}

\begin{proof}
The proof of Lemma \ref{Energy4} is quite close to the one of Lemma \ref{Energy3}, if one notices the fact in \eqref{U},
which implies
\begin{eqnarray}\|\mathcal{N}(z)\|_{\dot H^{-1}}\lesssim(1+\|z\|_{L^\infty})\|z\|_{W^{2,\infty}}\|z\|_{H^2}\lesssim \s^{-1-\delta}(\Xss+\Xsss),
\end{eqnarray}
consequently Lemma \ref{Energy4} is arrived by estimates under Duhamel formula similar as Lemma \ref{Energy3}.
\end{proof}

\subsection{Dispersive estimates}
We shall establishthe following lemma concerns asymptotic behavior of the solution:
\begin{lem}\label{Dispersive2}
Let $s$ be sufficiently large and $\delta$ be sufficiently small. Then the following inequality for any $T>0$:
\begin{eqnarray}\label{dispersive estimate 2}
\t^{1+\delta}\|(1-\varphi_\varepsilon)z\|_{W^{5,p}}+\t^{1+\delta}\|\varphi_\varepsilon z\|_{W^{5,q}}\leq C(\Xz+\XTT+\XTTT),
\end{eqnarray}
where the constant $C>0$ independent of $\varepsilon$.
\end{lem}

\begin{proof}
In light of \eqref{Duhamel}, we have for the linear part
\begin{eqnarray*}
\|(1-\varphi_\varepsilon)e^{iHt}z_0\|_{W^{5,p}}\lesssim \t^{-1-\delta}\|z_0\|_{W^{10,p'}\cap H^{10}};
\end{eqnarray*}
\begin{eqnarray*}
\|\varphi_\varepsilon e^{iHt}z_0\|_{W^{5,q}}\lesssim \t^{-1-\delta}\|\varphi_\varepsilon z_0\|_{W^{\frac{19}{2},q'}\cap H^{10}}\lesssim \t^{-1-\delta}\big(\varepsilon^\alpha\|\varphi_\varepsilon z_0\|_{W^{10,q'}}+\|z_0\|_{H^{10}}\big)
\end{eqnarray*}
for some $\alpha>0$. Therefore we are left with nonlinear integral, in which we still rewrite $\mathcal N(z)$ as \eqref{nonlinear 2}.

\subsubsection{Higher  order estimates}
Again we only focus on cubic terms $\C[a,a,a]:=U^{-1}\Lambda(1-\Delta)^{-1}\left((1-\Delta)^{-1}a\right)^3$. In fact, we directly have
 \begin{eqnarray}
\|(1-\varphi_\varepsilon)\int^{t}_{0}e^{iH(t-s)}\C[a,a,a]ds\|_{W^{5,p}}
\nonumber\lesssim\int^{t}_{0}\langle t-s\rangle^{-1-\delta}\|\C[a,a,a]\|_{W^{10,p'}\cap H^{10}}ds.
\end{eqnarray}
Now we only focus on $L^{p'}$ (or $L^{q'}$ later) estimates while Sobolev norm is much easier to control, where we could follow similar fashion as in subsection 3.3. Now we immediately have
$$\|\C[a,a,a]\|_{W^{10,p'}}\lesssim\|z\|^3_{W^{11,3p'}},$$
then inspired by that
$$\|z\|_{W^{11,3p'}}\lesssim
\|(1-\varphi_\varepsilon)z\|^{\frac{1+4\delta}{3+3\delta}}_{H^{s}}\|(1-\varphi_\varepsilon)z\|^{\frac{2-\delta}{3+3\delta}}_{W^{10,p}}+\|\varphi_\varepsilon z\|^{\frac{11+35\delta}{27+27\delta}}_{H^{s}}\|\varphi_\varepsilon z\|^{\frac{16-8\delta}{27+27\delta}}_{W^{10,q}}
\lesssim \s^{-\frac{(1+\delta)(16-8\delta)}{27+27\delta}}\Xs$$
which indicates
 \begin{eqnarray*}
\|(1-\varphi_\varepsilon)\int^{t}_{0}e^{iH(t-s)}\C[a,a,a]ds\|_{W^{5,p}}
\lesssim\int^{t}_{0}\langle t-s\rangle^{-1-\delta}\s^{-\frac{(1+\delta)(16-8\delta)}{9+9\delta}}\Xsss ds\lesssim\langle t\rangle^{-1-\delta}\Xttt.
\end{eqnarray*}
On the other hand, we also have
 \begin{eqnarray}
&&\|\varphi_\varepsilon\int^{t}_{1}e^{iH(t-s)}\C[a,a,a]ds\|_{W^{5,q}}\\
\nonumber&\lesssim&\int^{t}_{1}\langle t-s\rangle^{-1-\delta}\|\C[a,a,a]\|_{W^{10,q'}\cap H^{10}}ds
\lesssim\int^{t}_{1}\langle t-s\rangle^{-1-\delta}\|z\|^3_{W^{11,3q'}}ds.
\end{eqnarray}
Inspired by that
$$\|z\|_{W^{11,3q'}}\lesssim
\|(1-\varphi_\varepsilon)z\|^{\frac{3+11\delta}{8+8\delta}}_{H^{s}}\|(1-\varphi_\varepsilon)z\|^{\frac{5-3\delta}{8+8\delta}}_{W^{10,p}}+\|\varphi_\varepsilon z\|^{\frac{4+12\delta}{9+9\delta}}_{H^{s}}\|\varphi_\varepsilon z\|^{\frac{5-3\delta}{9+9\delta}}_{W^{10,q}}
\lesssim \s^{-\frac{(1+\delta)(5-3\delta)}{9+9\delta}}\Xs$$
which indicates
 \begin{eqnarray*}
\|\varphi_\varepsilon\int^{t}_{0}e^{iH(t-s)}\C[a,a,a]ds\|_{W^{5,q}}
\lesssim\int^{t}_{1}\langle t-s\rangle^{-1-\delta}\s^{-\frac{(1+\delta)(5-3\delta)}{3+3\delta}}\Xsss ds\lesssim\langle t\rangle^{-1-\delta}\Xttt.
\end{eqnarray*}
Quite similarly, we could handle those higher order terms.

\subsubsection{Quadratic estimates}

So we are left with quadratic terms, inspired by the non-vanishing property of $\hat\Omega_\varepsilon(\xi,\eta)$ and the normal form process \eqref{normal form 0}, we arrive at \eqref{normal form}. We begin with $W^{5,p}$ norm, where recall calculations in \eqref{Decay 111}, there holds
\begin{multline}\label{NHY}
\|(1-\varphi_\varepsilon)e^{iHt}e^{i\Omega t}\B_{1}[\Lambda^{-1}f,\Lambda^{-1}f]\|_{W^{5,p}}\leq\|(1-\varphi_\varepsilon)e^{iHt}e^{i\Omega t}\B_{2}[\Lambda^{-1}f,\Lambda^{-1}f]\|_{H^{5}}\\
\leq\|\Lambda^{-1}z\|_{W^{10,\tilde p}}\|\Lambda^{\frac{1}{2}-\frac{3}{p}}z\|_{W^{10,p}}\leq \t^{-1-\delta}\Xtt
\end{multline}
where we apply \eqref{VG} and for $\vartheta,\delta\ll1$
$$\|\Lambda^{\frac{1}{2}-\frac{3}{p}}z\|_{W^{10,p}}\lesssim\|(1-\varphi_\varepsilon)z\|_{W^{15,p}}+\|\varphi_\varepsilon z\|_{W^{15,q}}\lesssim\t^{-1-\frac{\delta}{2}}\Xt.$$
For initial interaction, we could similarly handled and here we omit those detailed estimates. Now we are left with time integral. Recall the
definition of $\tilde p$ in \eqref{delta p}, there holds
\begin{eqnarray*}
\nonumber\|(1-\varphi_\varepsilon)\B_{1}[e^{iHs}\partial_s f,z]\|_{W^{10,p'}}\lesssim\|\Lambda^{-1}e^{iHs}\partial_s f\|_{H^{15}}
\|\Lambda^{-1} z\|_{W^{15,\tilde p}}.
\end{eqnarray*}
In light of \eqref{VG} and the fact
\begin{eqnarray}\label{NNN}
\|\Lambda^{-1}\mathcal{N}(z)\|_{H^{15}}\lesssim(1+\|z\|_{L^\infty})
\|z\|_{H^{s}}\|z\|_{W^{20,\infty}}\lesssim \s^{-1-\frac{\delta}{2}}\Xss,
\end{eqnarray}
we get to
\begin{multline}
\|(1-\varphi_\varepsilon)e^{iH t}\int^{t}_{0} e^{i\Omega s}\B_{1}[\Lambda^{-1}\partial_{s}f,\Lambda^{-1}f]ds\|_{W^{5,p}}
\leq\int^{t}_0\langle t-s\rangle^{-1-\delta}\s^{-1-\delta}\Xsss ds\leq \t^{-1-\delta}\Xttt.
\end{multline}
This finish estimates in $W^{5,p}$ space where
$$\Big\|(1-\varphi_\varepsilon)\int^{t}_{0}e^{iH(t-s)}\B[\Lambda^{-1}z^\pm,\Lambda^{-1}z^\pm]ds\Big\|_{W^{5,p}}\lesssim\t^{-1-\delta}\big(\mathcal{X}^2_0+\Xtt+\Xttt\big)$$

Next for $W^{5,q}$ norm, since that if $\xi\in\mathrm{supp}\hat\varphi_\varepsilon(\xi)$, then $1\leq|\xi|^{\tau},\tau>0$, there holds
\begin{eqnarray}
\|\varphi_\varepsilon e^{iHt}e^{i\Omega t}\B_{1}[\Lambda^{-1}f,\Lambda^{-1}f]\|_{W^{5,q}}\leq\|\varphi_\varepsilon \Lambda^{5}e^{iHt}e^{i\Omega t}\B_{1}[\Lambda^{-1}f,\Lambda^{-1}f]\|_{ H^{5}}.
\end{eqnarray}
Hence repeating calculations as in \eqref{NHY} yields desired estimates. On the other hand, for initial data's interaction, we have
\begin{eqnarray*}
\|\varphi_\varepsilon e^{iH t}\B_{1}[\Lambda^{-1}z_0,\Lambda^{-1}z_0]\|_{W^{5,q}}
&\leq&\t^{-1-\delta}\|\B_{1}[\Lambda^{-1}z_0,\Lambda^{-1}z_0]\|_{W^{10,q'}\cap H^{10}}\\&\leq&\t^{-1-\delta}\|z_0\|^2_{\dot H^{-1}\cap H^{15}}\leq\t^{-1-\delta}\mathcal{X}^2_0.
\end{eqnarray*}
Finally for time integral, utilizing Lemma \ref{Multiplier estimates}, we select 
\begin{eqnarray}\label{delta q}\tilde q=\frac{24}{11}-f_2(\delta,\vartheta),\quad \mathrm{where} \,\,\,f_2(\delta,\vartheta)=\frac{24(9\delta+8\vartheta)}{11+9\delta+8\vartheta},\end{eqnarray}
there holds
\begin{eqnarray*}
\nonumber\|\varphi_\varepsilon\B_{1}[e^{iHs}\partial_s f,z]\|_{W^{10,q'}}\lesssim\|\Lambda^{-1}e^{iHs}\partial_s f\|_{H^{15}}
\|\Lambda^{-1} z\|_{W^{15,\tilde q}}.
\end{eqnarray*}
In light of \eqref{NNN} and the fact
\begin{eqnarray}\|\Lambda^{-1}z\|_{W^{15,\tilde q}}\lesssim\|\Lambda^{-1}z\|^{\frac{2}{\tilde q}}_{H^{s}}\|\Lambda^{-1}z\|^{1-\frac{2}{\tilde q}}_{W^{5,\infty}}\lesssim\s^{-\frac{\delta}{2}}\Xt,\end{eqnarray}
we finally obtain
\begin{eqnarray}
\|\varphi_\varepsilon\B_{1}[e^{iHs}\partial_s f,z]\|_{W^{10,q'}}\lesssim \s^{-1-\delta}\Xss
\end{eqnarray}
and naturally
\begin{eqnarray*}\|\varphi_\varepsilon e^{iH t}\int^{t}_{0} e^{i\Omega s}\B_{1}[\Lambda^{-1}\partial_{s}f,\Lambda^{-1}f]ds\|_{W^{5,q}}&\lesssim&\int^{t}_{0}\langle t-s\rangle^{-1-\delta}\|\varphi_\varepsilon\B_{1}[e^{iHs}\partial_s f,z]\|_{W^{10,q'}\cap H^{10}}ds\\&\lesssim&\int^{t}_0\langle t-s\rangle^{-1-\delta}\s^{-1-\delta}\Xsss ds\leq \t^{-1-\delta}\Xttt,
\end{eqnarray*}
which finish the proof of Lemma \ref{Dispersive2}.
\end{proof}

Lemma \ref{Energy3}, \ref{Energy4} and \ref{Dispersive2} finally lead to Proposition \ref{priori2}. Then the Well-posedness and scattering exactly follow the proof of Theorem 
\ref{thm1} and this completes the Theorem \ref{thm2}.

\section{Proof of Theorem \ref{thm3}}\setcounter{equation}{0}
\hspace*{\fill}

In this section, we study the high Mach number limit of solutions we established in the previous sections. Denote the error by $(\tilde a,\tilde u)=(a^{\varepsilon}-a^{0},u^{\varepsilon}-u^{0 })$, $\tilde{z}=z^\varepsilon-z^0$, $\tilde f=f^\varepsilon-f^0$, then by Duhamel formula, we immediately reach for the error of profile $\tilde{f}$ satisfying
\begin{eqnarray}
\tilde f=\tilde{z}_{0}+\int^{t}_{0}e^{iH_{\varepsilon}s}\mathcal{N}_{\varepsilon}(z^{\varepsilon})ds-\int^{t}_{0}e^{iH_{0}s}{\mathcal{N}}_{0}(z^{0})ds
\end{eqnarray}
where $\tilde{z}_{0}=z^{\varepsilon}_{0}-z^0_{0}$. Hence, first of all, taking $L^2$ norm implies
\begin{eqnarray}
\|\tilde f\|_{L^2}\lesssim\|\tilde{z}_{0}\|_{L^2}+\|\int^{t}_{0}\left(e^{iH_{\varepsilon}s}\mathcal{N}_{\varepsilon}(z^{\varepsilon})-e^{iH_{0}s}{\mathcal{N}}_{0}(z^{0})\right)ds\|_{L^2}.
\end{eqnarray}

For nonlinear integral, we shall decompose the time variable $s$ for $s\leq\varepsilon^{-\frac{1}{2}}$ or $s\geq\varepsilon^{-\frac{1}{2}}$ by inserting the localized function $\chi(\frac{s}{\varepsilon^{-1/2}})$. In the situation of $s\geq\varepsilon^{-\frac{1}{2}}$, we utilize the property $1\leq s^{\gamma}\varepsilon^{\frac{\gamma}{2}}$ and applying Lemma \ref{aaaa} to get for $\gamma<\frac{\delta}{2}$ that
\begin{eqnarray}\label{ttuu}
&&\|\int^{t}_{0}\big(1-\chi(\frac{s}{\varepsilon^{-1/2}})\big)\left(e^{iH_\varepsilon s}\mathcal{N}_\varepsilon(z^\varepsilon)-e^{iH_{0}s}{\mathcal{N}}_{0}(z^0)ds\right)\|_{L^2}\\
\nonumber&\lesssim&\varepsilon^{\frac{\gamma}{2}}\int^{t}_{0}s^{\gamma}G(\|z^0,z^\varepsilon\|_{L^\infty})\|(z^0,z^\varepsilon)\|_{W^{2,\infty}}\|(z^0,z^\varepsilon)\|_{H^{s}}ds\\
\nonumber&\lesssim&\varepsilon^{\frac{\gamma}{2}}(\Xtt+\Xttt)\int^{t}_{0}s^{\gamma}\s^{-1-\delta}ds\leq \varepsilon^{\frac{\gamma}{2}}
\end{eqnarray}
where we  take $\gamma\ll\delta$. Next we consider $s\leq\varepsilon^{-\frac{1}{2}}$. One could further write
\begin{eqnarray*}
&&\int^{t}_{0}\chi(\frac{s}{\varepsilon^{-1/2}})\left(e^{iHs}\mathcal{N}_\varepsilon(z^\varepsilon)ds-e^{iH_{0}s}{\mathcal{N}}_{0}(z^0)\right)ds\\
&=&\int^{t}_{0}\chi(\frac{s}{\varepsilon^{-1/2}})e^{iH_{0}s}F_0(e^{-iH_{0}s}(f^0-e^{iH_\varepsilon s}z^\varepsilon)+e^{-iH_{0}s}e^{iH_\varepsilon s}z^\varepsilon)ds\\
&-&\int^{t}_{0}\chi(\frac{s}{\varepsilon^{-1/2}})e^{iH_\varepsilon s}F_\varepsilon(z^\varepsilon)ds-\varepsilon\int^{t}_{0}\chi(\frac{s}{\varepsilon^{-1/2}})e^{iH_\varepsilon s}G_\varepsilon(z^\varepsilon)ds=\sum^{4}_{i=1}C_{i}
\end{eqnarray*}
where
$$C_{1}=\int^{t}_{0}\chi(\frac{s}{\varepsilon^{-1/2}})e^{iH_{0}s}\left(F_0 (e^{-iH_{0}s}(\tilde f+f^\varepsilon))-F_0(e^{-iH_{0}s}f^\varepsilon)\right)ds;$$
$$C_{2}=\int^{t}_{0}\chi(\frac{s}{\varepsilon^{-1/2}})e^{iH_{0}s}F_0(e^{-iH_{0}s}f^\varepsilon)ds-\int^{t}_{0}\chi(\frac{s}{\varepsilon^{-1/2}})e^{iH_{0}s}F_{\varepsilon}(e^{-iH_{0}s}f^\varepsilon)ds;$$
$$C_{3}=\int^{t}_{0}\chi(\frac{s}{\varepsilon^{-1/2}})e^{iH_{0}s}F_{\varepsilon}(e^{-iH_{0}s}f^\varepsilon)ds-\int^{t}_{0}\chi(\frac{s}{\varepsilon^{-1/2}})e^{iH_\varepsilon s}F_{\varepsilon}(e^{-iH_\varepsilon s}f^\varepsilon)ds;$$
$$C_{4}=-\varepsilon\int^{t}_{0}\chi(\frac{s}{\varepsilon^{-1/2}})e^{iH_\varepsilon s}G_\varepsilon(z^\varepsilon)ds.$$

Our task is to bound $C_{i}$, where for $C_{2},C_{3},C_{4}$, we expect to obtain  with initial error or  some smallness of $\varepsilon$. In fact, inspired by $s\leq\varepsilon^{-\frac{1}{2}}$, we directly handle $C_{4}$ by
\begin{eqnarray}\label{small kappa}
\|C_{4}\|_{L^2}\lesssim\varepsilon\int^{t}_{0}\chi(\frac{s}{\varepsilon^{-1/2}})(1+\|z^\varepsilon\|_{L^\infty})\|z^\varepsilon\|_{H^s}\|z^\varepsilon\|_{W^{5,\infty}}ds\lesssim\varepsilon.
\end{eqnarray}

As for $C_{3}$, we consider the quadratic interaction of $\B[z^\varepsilon,z^\varepsilon]$ and write in Fourier space such that
\begin{eqnarray}
\nonumber&&\int^{t}_{0}e^{iH_{0}s}\B[e^{-iH_{0}s}f^\varepsilon,e^{-iH_{0}s}f^\varepsilon]ds-\int^{t}_{0}e^{iH_\varepsilon s}\B[e^{-iH_\varepsilon s}f^\varepsilon,e^{-iH_\varepsilon s}f^\varepsilon]ds\\
\nonumber&=&\mathcal{F}^{-1}\big(\int^{t}_{0}(e^{i\hat\Omega_{0,3}s}-e^{i\hat\Omega_{\varepsilon,3}s})B(\xi,\eta)\hat{f^\varepsilon}(\xi-\eta)\hat{f^\varepsilon}(\eta)d\eta ds\big)\\
\nonumber&=&\mathcal{F}^{-1}\big(\int^{t}_{0}\int^{1}_{0}\frac{d}{d\theta}e^{i\hat\Omega_{\theta,3}s}d\theta B(\xi,\eta)\hat{f^\varepsilon}(\xi-\eta)\hat{f^\varepsilon}(\eta) d\eta ds\big)\\
&=&i\varepsilon\int^{t}_{0}s\int^{1}_{0}e^{i\hat\Omega_{\theta,3}s}\widehat{B_{\theta}}[f^\varepsilon,f^\varepsilon]d\theta d\eta ds
\end{eqnarray}
where the symbol $\hat B_{\theta}$ is given by
$$\hat B_{\theta}(\xi,\eta)= B(\xi,\eta)\Big(|\xi|^2\H^{-1}_{\theta\varepsilon}(|\xi|)\pm|\xi-\eta|^2\H^{-1}_{\theta\varepsilon}(|\xi-\eta|)\pm|\eta|^2\H^{-1}_{\theta\varepsilon}(|\eta|)\Big)$$
and
$$\hat\Omega_{\theta,3}(\xi,\eta)=\H_{\theta\varepsilon}(|\xi|)\pm\H_{\theta\varepsilon}(|\xi-\eta|)\pm\H_{\theta\varepsilon}(|\eta|).$$
Since that there holds for $|B_{\theta}(\xi,\eta)|\leq \langle M^3\rangle$, then applying bilinear estimates yield
\begin{eqnarray}
\|C_{3}\|_{L^2}\lesssim\varepsilon\int^{t}_{0}\chi(\frac{s}{\varepsilon^{-1/2}})s\|e^{-iH_{\theta\varepsilon}s}f^\varepsilon\|_{W^{3,\infty}}\|z\|_{H^3}ds.
\end{eqnarray}
Now we claim the following inequality holds:
\begin{eqnarray}\label{error profile}
\|e^{-iH_{\theta\varepsilon}s} f^\varepsilon\|_{W^{3,\infty}}\lesssim\s^{-1-\delta}+\varepsilon s,\quad \mathrm{for\,\, all} \,\,\theta\in[0,1].
\end{eqnarray}
In fact, we could write
$$e^{-iH_{\theta\varepsilon}s} f^\varepsilon=z^\varepsilon+(e^{-iH_{\theta\varepsilon}s}-e^{-iH_{\varepsilon}s})f^{\varepsilon},$$
then  it is sufficient to  bound the second one where we consider it in the Fourier space such that
\begin{eqnarray}
\mathcal{F}((e^{-iH_{\theta\varepsilon}s}-e^{-iH_{\varepsilon}s})f^\varepsilon)&=&(e^{-i\hat H_{\theta\varepsilon}(\xi)s}-e^{-i\hat H_{\varepsilon}(\xi)s})\hat f^\varepsilon(\xi)\\
\nonumber&=&\int^{1}_{\theta}\frac{d}{d\sigma}e^{-i\hat H_{\sigma\varepsilon}(\xi)s}\hat f^\varepsilon(\xi)d\sigma\\
\nonumber&=&-i\varepsilon s\int^{1}_{\theta}e^{-i\hat H_{\sigma\varepsilon}(\xi)s}|\xi|^2\hat f^\varepsilon(\xi)d\sigma,
\end{eqnarray}
which implies
\begin{eqnarray}\label{FFFVVV}
\|(e^{-iH_{\theta\varepsilon}s}-e^{-iH_{\varepsilon}s})f^\varepsilon\|_{H^{5}}\lesssim\varepsilon s\|f^\varepsilon\|_{H^{10}}\lesssim\varepsilon s.
\end{eqnarray}
Hence \eqref{error profile} immediately follows by above inequality and the fact $\|z^\varepsilon\|_{W^{3,\infty}}\lesssim\s^{-1-\delta}$. In light of \eqref{error profile}, we obtain
\begin{eqnarray}
\|C_{3}\|_{L^2}\lesssim\varepsilon\int^{t}_{0}\chi(\frac{s}{\varepsilon^{-1/2}})\s^{-\delta}ds+\varepsilon^2\int^{t}_{0}\chi(\frac{s}{\varepsilon^{-1/2}})s^2ds\lesssim\varepsilon^{\frac{1}{2}}.
\end{eqnarray}
Other higher order terms in $C_{3}$ could be handled similarly.

For $C_{2}$, we also expect to obtain some smallness of $\varepsilon$ from error. Indeed, we take the convection term $\mathcal{Q}_{\varepsilon}[f,g]=U_{\varepsilon}^{-1}\Lambda^{-1}{\div}\div(\nabla U_\varepsilon\Lambda^{-1}f\otimes\nabla U_\varepsilon\Lambda^{-1}g)$ as an example, there holds
\begin{multline*}
\mathcal{F}({\mathcal{Q}_{\varepsilon}[f,g]-\mathcal{Q}_{0}[f,g]})(\xi)\\=\int_{\R^3}Q(\xi,\eta)\Big(\frac{\U_\varepsilon(|\xi-\eta|)\U_\varepsilon(|\eta|)}{\U_\varepsilon(|\xi|)}-\frac{\U_0(|\xi-\eta|)\U_0(|\eta|)}{\U_0(|\xi|)}\Big) \widehat{f}(\xi-\eta)\hat g(\eta) d\eta\\=\int_{\R^3}Q(\xi,\eta)\int^1_0\frac{d}{d\theta}\frac{\U_{\theta\varepsilon}(|\xi-\eta|)\U_{\theta\varepsilon}(|\eta|)}{\U_{\theta\varepsilon}(|\xi|)}d\theta \widehat{f}(\xi-\eta)\hat g(\eta) d\eta\\
=\varepsilon\int^{1}_{0}\int_{\R^3}Q_\theta(\xi,\eta)\widehat{f}(\xi-\eta)\widehat {g}(\eta) d\eta d\theta
=\varepsilon\int^{1}_{0}\mathcal{Q}_{\theta}[f,g]d\theta,
\end{multline*}
where $Q(\xi,\eta)=\frac{\xi\cdot\xi\cdot((\xi-\eta)\eta^T)}{|\xi||\xi-\eta||\eta|}$ and
$$Q_\theta(\xi,\eta)=Q(\xi,\eta)\Big(\frac{\U_{\theta\varepsilon}(|\eta|)}{\U_{\theta\varepsilon}(|\xi-\eta|)\U_{\theta\varepsilon}(|\xi|)}+\frac{\U_{\theta\varepsilon}(|\xi-\eta|)}{\U_{\theta\varepsilon}(|\eta|)\U_{\theta\varepsilon}(|\xi|)}+\frac{\U_{\theta\varepsilon}(|\xi-\eta|)\U_{\theta\varepsilon}(|\eta|)}{\U^3_{\theta\varepsilon}(|\xi|)}\Big).$$
Therefore, make use of \eqref{error profile} and the pointwise estimates:
$$|\UU(r)|\lesssim1; |\UU^{-1}(r)|\lesssim\langle r\rangle \quad\Rightarrow\quad|\nabla_\xi^\alpha\nabla_\eta^\beta Q_\theta(\xi,\eta)|\lesssim\langle M^3\rangle|\xi|^{-\alpha}|\eta|^{-\beta},$$
we deduct
\begin{eqnarray*}
&&\|\int^{t}_{0}\chi(\frac{s}{\varepsilon^{-1/2}})\mathcal{Q}_{\varepsilon}[e^{-iH_{0}s}f^\varepsilon,e^{-iH_{0}s}f^\varepsilon]ds\|_{L^2}\\
&\lesssim&\varepsilon\int^{t}_{0}\chi(\frac{s}{\varepsilon^{-1/2}})\|\int^{1}_{0}\mathcal{Q}_{\theta}[e^{-iH_{0}s}f^\varepsilon,e^{-iH_{0}s}f^\varepsilon]\|_{L^2}d\theta ds\\
&\lesssim&\varepsilon\int^{t}_{0}\chi(\frac{s}{\varepsilon^{-1/2}})\|e^{-iH_{\theta\varepsilon}s}f^\varepsilon\|_{W^{3,\infty}}\|z\|_{H^3} ds\lesssim\varepsilon^{\frac{1}{2}}.
\end{eqnarray*}
$C_{2}$'s other nonlinearities can be controlled in a similar way.

Finally we control $C_{1}$. We still consider the quadratic interaction of $\B[z^\varepsilon,z^\varepsilon]$, in which one could write
\begin{eqnarray*}
&&\int^{t}_{0}\chi(\frac{s}{\varepsilon^{-1/2}})e^{iH_{0}s}\left(\B[e^{-iH_{0}s}(\tilde f+f^\varepsilon),e^{-iH_{0}s}(\tilde f+f^\varepsilon)]-\B[e^{-iH_{0}s}f^\varepsilon,e^{-iH_{0}s}f^\varepsilon]\right)ds\\
&=&\int^{t}_{0}\chi(\frac{s}{\varepsilon^{-1/2}})e^{iH_{0}s}\B[e^{-iH_{0}s}\tilde f,e^{-iH_{0}s}f^\varepsilon]ds+\int^{t}_{0}\chi(\frac{s}{\varepsilon^{-1/2}})e^{iH_{0}s}\B[e^{-iH_{0}s}\tilde f,e^{-iH_{0}s}\tilde f]ds.
\end{eqnarray*}
Directly applying product estimates leads to
\begin{eqnarray*}
&&\|\int^{t}_{0}\chi(\frac{s}{\varepsilon^{-1/2}})e^{iH_{0}s}\left(\B[e^{-iH_{0}s}(\tilde f+f^\varepsilon),e^{-iH_{0}s}(\tilde f+f^\varepsilon)]-\B[e^{-iH_{0}s}f^\varepsilon,e^{-iH_{0}s}f^\varepsilon]\right)ds\|_{L^2}\\
&\lesssim&\int^{t}_{0}\chi(\frac{s}{\varepsilon^{-1/2}})(\|\tilde f\|_{L^2}+\|\tilde f\|_{\dot{H}^{2}})(\|e^{-iH_{0}s}f^\varepsilon\|_{W^{1,\infty}}+\|z^0\|_{W^{1,\infty}})ds\\
&\lesssim&\int^{t}_{0}\chi(\frac{s}{\varepsilon^{-1/2}})(\|\tilde f\|_{L^2}+\|\tilde f\|_{\dot{H}^{2}})(\s^{-1-\delta}+\varepsilon s)\Xs ds\leq \Xt(\|\tilde f\|_{L^2}+\sup_{t\leq\varepsilon^{-\frac{1}{2}}}\|\tilde f\|_{\dot{H}^{2}}).
\end{eqnarray*}
Now there holds
\begin{eqnarray}\label{rrrt}
&&\sup_{t\leq\varepsilon^{-\frac{1}{2}}}\|\tilde f \|_{\dot{H}^2}\\&=&\nonumber\sup_{t\leq\varepsilon^{-\frac{1}{2}}}\Big\|e^{-iH_0t}(\tilde a,U^{-1}_0\Lambda \tilde \psi)+e^{-iH_0t}(U^{-1}_\varepsilon-U^{-1}_0)\Lambda\psi^\varepsilon+(e^{-iH_\varepsilon t}-e^{-iH_{0}t})z^\varepsilon\Big\|_{\dot{H}^2}\\
&\leq&\nonumber\sup_{t\leq\varepsilon^{-\frac{1}{2}}}\Big(\|(\tilde{a},U^{-1}_0\Lambda\tilde{u})\|_{\dot{H}^2}+
\|(U^{-1}_\varepsilon-U^{-1}_0)\Lambda\psi^\varepsilon\|_{\dot{H}^2}+\|(e^{-iH_\varepsilon t}-e^{-iH_{0}t})z^\varepsilon\|_{\dot{H}^2}\Big)
\end{eqnarray}
where we utilize $z^\varepsilon-z^0=\tilde a+U^{-1}_\varepsilon\Lambda\tilde{\psi}+(U^{-1}_\varepsilon-U^{-1}_0)\Lambda\psi^0$.
Now we claim the following inequalities:
\begin{eqnarray}\label{local energy}
\sup_{t\leq\varepsilon^{-\frac{1}{2}}}\|(\tilde{a},U^{-1}_0\Lambda\tilde{\psi})\|_{\dot{H}^2}\leq C(\|(\tilde{a}_{0},\tilde{u}_{0})\|_{H^{4}}+\varepsilon^{\frac{1}{2}}\Xt),
\end{eqnarray}
\begin{eqnarray}\label{mm11}
\|(U^{-1}_\varepsilon-U^{-1}_0)\Lambda\psi^\varepsilon\|_{\dot{H}^2}\lesssim\varepsilon\Xt.
\end{eqnarray}
We postpone the proof of (\ref{local energy}) at the end of this section. For \eqref{mm11}, we consider it in the Fourier space such that
\begin{eqnarray}
\mathcal{F}((U^{-1}_\varepsilon-U^{-1}_0)\Lambda\psi^\varepsilon)&=&(\hat U^{-1}_\varepsilon(\xi)-\hat U^{-1}_0(\xi))|\xi|\hat{\psi^\varepsilon}(\xi)\\
\nonumber&=&\int^{1}_{0}\frac{d}{d\theta}U^{-1}_{\theta\varepsilon}(\xi)d\theta|\xi|\hat{\psi^\varepsilon}(\xi)\\
\nonumber&=&-\varepsilon \int^{1}_{0}U^{-3}_{\theta\varepsilon}(\xi)d\theta|\xi|\hat{\psi^\varepsilon}(\xi).
\end{eqnarray}
Since that $|U^{-3}_{\theta\varepsilon}(\xi)|\leq\langle |\xi|^3\rangle$, we have
\begin{eqnarray}
\|(U^{-1}_\varepsilon-U^{-1}_0)\Lambda\psi^\varepsilon\|_{\dot H^2}\leq\varepsilon\|z^\varepsilon\|_{H^5}\leq \varepsilon\Xt.
\end{eqnarray}
In light of \eqref{FFFVVV} and (\ref{rrrt})-\eqref{mm11}, we have
\begin{eqnarray*}
\|C_1\|_{L^2}
&\lesssim&\Xt\|\tilde f\|_{L^2}+\sup_{t\leq\varepsilon^{-\frac{1}{2}}}\big(\|(\tilde{a}_{0},\tilde{u}_{0})\|_{H^{4}}+\varepsilon^{\frac{1}{2}}\Xt+\varepsilon t\big)\\
&\lesssim&\Xt\|\tilde f\|_{L^2}+\|(\tilde{a}_{0},\tilde{u}_{0})\|_{H^{4}}+\varepsilon^{\frac{1}{2}}.
\end{eqnarray*}
Finally, notice that $\X\ll1$, combining above inequality and \eqref{ttuu}, we are able to conclude with
\begin{eqnarray}
\|\tilde f\|_{L^2}\lesssim\|(\tilde{a}_{0},\tilde{u}_{0})\|_{H^{4}}+\varepsilon^{\frac{\gamma}{2}}\lesssim\|(\tilde{a}_{0},\tilde{u}_{0})\|^{1-\frac 4s}_{L^2}(\mathcal{X}^\varepsilon_0+\mathcal{X}^0_0)^{\frac 4s}+\varepsilon^{\frac{\gamma}{2}}.
\end{eqnarray}
Therefore we find that $\|\tilde f\|_{L^2}\Rightarrow0$ provided $\|(\tilde{a}_{0},\tilde{u}_{0})\|_{L^2}\Rightarrow0$ and $\varepsilon\Rightarrow0$. Convergence for regularity space $\dot{H}^N$ for $0<N<s$ is a natural results of interpolation and this conclude with proof of Theorem \ref{thm3}.

\subsection{Controlling of local energy}

In this subsection we come to prove (\ref{local energy}), our strategy is to apply the energy, which shall enable us to overcome the derivative loss problem. Let us recall that the pressureless (iEP) and (iEP) could be given as
\begin{equation}
\left\{
\begin{array}{l}\partial_{t}a^0+u^0\cdot\nabla a^0+\div u^0=-a^0\div u^0,\\ [1mm]
\partial_{t}u^0+u^0\cdot\nabla u^0 +(1-\Delta)^{-1}\nabla a^0=\nabla R(a^0),\\[1mm]
 \end{array} \right.
\end{equation}

\begin{equation}\label{AB2}
\left\{
\begin{array}{l}\partial_{t}a^\varepsilon+\div u^\varepsilon=-\mathrm{div}(a^\varepsilon u^\varepsilon),\\ [1mm]
\partial_{t}u^\varepsilon+u^\varepsilon\cdot\nabla u^\varepsilon+\varepsilon  \nabla  P(a^\varepsilon)+(1-\Delta)^{-1}\nabla a^\varepsilon=\nabla R(a^\varepsilon),\\[1mm]
 \end{array} \right.
\end{equation}
Naturally, the error $(\tilde{a},\tilde{u})=(a^{\varepsilon},u^{\varepsilon})-(a^{0},u^{0})$ can be presented as
\begin{equation}
\left\{
\begin{array}{l}\partial_{t}\tilde{a}+\div\tilde u+\tilde{u}\cdot \nabla a^{\varepsilon}+u^{0}\cdot \nabla \tilde{a}+\tilde a\mathrm{div}u^{\varepsilon}+a^0\mathrm{div}\tilde{u}=0,\\ [1mm]
 \partial_{t}\tilde{u}+\tilde{u}\cdot \nabla u^{\varepsilon}+u^{0}\cdot \nabla \tilde{u}+\varepsilon  \nabla  P(a^\varepsilon)+(1-\Delta)^{-1}\nabla \tilde a=0.\\[1mm]
 \end{array} \right.
\end{equation}

Then standard energy as in Section 3 calls that
\begin{multline}
\frac{d}{dt}\|(\Lambda^\ell \tilde a,\mathcal{A}^\ell \tilde u)\|_{L^2}
+\int_{\R^3}\Lambda^\ell(\tilde{u}\cdot \nabla a^{\varepsilon}+u^{0}\cdot \nabla \tilde{a}+\tilde a\mathrm{div}u^{\varepsilon}+a^0\mathrm{div}\tilde{u})\Lambda^\ell\tilde{a}dx\\
+\int_{\R^3}\mathcal{A}^\ell(\tilde{u}\cdot \nabla u^{\varepsilon}+u^{0}\cdot \nabla \tilde{u}+\varepsilon  \nabla  P(a^\varepsilon))\mathcal{A}^\ell\tilde{u}dx=0.
\end{multline}
We start with those terms with no derivative loss, for example $\tilde{u}\cdot \nabla a^{\varepsilon}$. It holds
\begin{eqnarray}
\int_{\R^3}\Lambda^\ell(\tilde{u}\cdot \nabla a^{\varepsilon})\Lambda^\ell\tilde{a}dx
\lesssim\|\nabla a^{\varepsilon}\|_{W^{\ell,\infty}}\|\Lambda^\ell(\tilde a,\tilde u)\|^{2}_{L^{2}}.
\end{eqnarray}
On the other hand, we pay attention to transport terms
\begin{multline}
\int_{\R^3}\Lambda^\ell(u^{0}\cdot \nabla \tilde{a})\Lambda^\ell\tilde{a}dx=\int_{\R^3}[\Lambda^\ell,u^{0}\cdot \nabla ]\tilde{a}\Lambda^\ell\tilde{a}dx
+\int_{\R^3}u^{0}\cdot \Lambda^\ell\nabla\tilde{a}\Lambda^\ell\tilde{a}dx\\
\lesssim\|\nabla u^{0}\|_{W^{\ell,\infty}}\|\Lambda^\ell(\tilde a,\tilde u)\|^{2}_{L^{2}}.
\end{multline}
Similarly, we could handle other terms involving $\Lambda^\ell$ or $\mathcal{A}^\ell$ by utilizing Lemma \ref{eeee} and \ref{commutator}, which yields
\begin{multline}
\int_{\R^3}\Lambda^\ell(\tilde{u}\cdot \nabla a^{\varepsilon}+u^{0}\cdot \nabla \tilde{a}+\tilde a\mathrm{div}u^{\varepsilon}+a^0\mathrm{div}\tilde{u})\Lambda^\ell\tilde{a}dx
+\int_{\R^3}\mathcal{A}^\ell(\tilde{u}\cdot \nabla u^{\varepsilon}+u^{0}\cdot \nabla \tilde{u})\mathcal{A}^\ell\tilde{u}dx\\
\lesssim\|(a^0,a^\varepsilon, u^{0},u^\varepsilon)\|_{W^{\ell+1,\infty}}\|(\tilde a,\tilde u)\|^{2}_{H^{\ell+1}}\lesssim \t^{-1-\delta}\Xt\|(\tilde a,\tilde u)\|^{2}_{H^{\ell+1}}
\end{multline}
provided $\ell\leq 4\ll s$.

Hence we are left with the pressure term.
For $\varepsilon\int_{\R^3}\mathcal{A}^\ell\nabla  P(a^\varepsilon)\mathcal{A}^\ell\tilde{u}dx$. We just treat the linear term $\varepsilon\int_{\R^3}\mathcal{A}^\ell\nabla a^\varepsilon\mathcal{A}^\ell\tilde{u}dx$ where it holds
\begin{eqnarray*}
\varepsilon\int_{\R^3}\mathcal{A}^\ell\nabla a^\varepsilon\mathcal{A}^\ell\tilde{u}dx\lesssim\varepsilon\|a^{\varepsilon}\|_{H^{\ell+2}}\|\tilde{u}\|_{H^{\ell+1}}\lesssim\varepsilon(\|a^{\varepsilon}\|^2_{H^{\ell+2}}+\|\tilde{u}\|^2_{H^{\ell+1}}).
\end{eqnarray*}
Hence based on fact $\|a^{\varepsilon}\|_{H^{\ell+2}}\lesssim c$, one could conclude with
\begin{eqnarray}
\frac{d}{dt}\|(\Lambda^\ell \tilde a,\mathcal{A}^\ell \tilde u)\|^{2}_{L^2}\leq c\varepsilon+\left(\varepsilon+\t^{-1-\delta}\Xt\right)\|(\tilde a,\tilde u)\|^{2}_{H^{\ell+1}}.
\end{eqnarray}
Denote $\E=\sup\limits_{T\in[0,\varepsilon^{-\frac{1}{2}}]}\|(\Lambda^\ell \tilde a,\mathcal{A}^\ell \tilde u)(T)\|_{L^2}$, there immediately yields for $\ell=3$ such that
\begin{eqnarray*}
\E^{2}&\leq&\|(\tilde{a}_{0},\tilde{u}_{0})\|^{2}_{H^{4}}+\int^{T}_{0}\big(c{\varepsilon}+\left(\varepsilon+\t^{-1-\delta}\right)\|(\tilde{a},\tilde{u})\|^{2}_{H^{4}}\big)dt\\
&\leq&\|(\tilde{a}_{0},\tilde{u}_{0})\|^{2}_{H^{4}}+c\varepsilon^{\frac{1}{2}}+c_{\mu}\E^{2},
\end{eqnarray*}
where $0<c_{\mu}\ll1$, we naturally have
$$\sup_{T\leq\varepsilon^{-\frac{1}{2}}}\|(\tilde{a},U^{-1}_0\Lambda\tilde{\psi})\|_{\dot{H}^2}\leq\|(\tilde{a},\tilde{u})\|_{H^{3}}\lesssim \|(\tilde{a}_{0},\tilde{u}_{0})\|^{2}_{H^{4}}+\varepsilon^{\frac{1}{2}}$$
and this compete the proof of (\ref{local energy}).

\vspace{6pt}

\section{Appendix}

\subsection{Littlewood-Paley theory}

For the reader's convenience, we briefly review the basic framework of Fourier localization and the Littlewood--Paley theory that will be used throughout the paper. Standard references include Chapters~2--3 of \cite{BCD}.


Let $\chi\in C_c^\infty(\mathbb{R}^d)$ be a radial function satisfying $0\le \chi\le1$ and $ \mathrm{supp}\,\chi\subset \{\xi:|\xi|\le 4/3\}$.
Define
\[
\varphi(\xi)=\chi(\xi/2)-\chi(\xi),
\]
so that $\varphi$ is supported in the annulus $\{\xi\in\mathbb{R}^d:3/4\le |\xi|\le 8/3\}$ and $\sum_{q\in\mathbb{Z}}\varphi(2^{-q}\xi)=1$ for $\xi\neq0$.

For any tempered distribution $f\in\mathcal{S}'$, we introduce the homogeneous dyadic blocks $\dot{\Delta}_j $ and  the low-frequency cut-off operators $\dot{S}_j$ by
\[
\dot{\Delta}_j f=\varphi(2^{-q}D)f\quad\text{and}\quad\dot{S}_j f=\chi(2^{-q}D)f \qquad q\in\mathbb{Z},
\]
where $\varphi(2^{-q}D)$ and $\chi(2^{-q}D)$ are defined as Fourier multipliers with symbols $\varphi(2^{-q}\xi)$ and $\chi(2^{-q}\xi)$, respectively.

Let $\mathcal{P}$ denote the space of polynomials and set $\mathcal{S}'_0=\mathcal{S}'/\mathcal{P}$. Then any $f\in\mathcal{S}'_0$ admits the homogeneous Littlewood–Paley decomposition
\[
f=\sum_{q\in\mathbb{Z}}\dot{\Delta}_j f \quad \text{in }\quad \mathcal{S}'_0 .
\]
Given the threshold $j_0$, we further define the low- and high-frequency parts of $f$ as follows:
\[
f^\ell=\dot{S}_{j_0}f=\sum_{j\leq j_0-1}\dot{\Delta}_j f,\qquad 
f^h=({\rm Id}-\dot{S}_{j_0})f=\sum_{j\geq j_0}\dot{\Delta}_j f .
\]

\begin{defn}\label{defn2.1}
Let $s\in\mathbb{R}$ and $1\le p,r\le\infty$.  
The homogeneous Besov space $\dot{B}^s_{p,r}$ consists of all $f\in\mathcal{S}'_0$ such that
\[
\|f\|_{\dot{B}^s_{p,r}}
=\Big\|\Big\{2^{js}\|\dot{\Delta}_jf\|_{L^p}\Big\}_{j\in\mathbb{Z}}\Big\|_{l^{r}}.
\]
\end{defn}

We recall several classical properties of homogeneous Besov spaces (see \cite{BCD}):

\medskip
\noindent
$\bullet$ \emph{Scaling.}  
For any $\sigma\in\mathbb{R}$ and $1\le p,r\le\infty$, there exists $C>0$ such that for all $\lambda>0$,
\[
\|f(\lambda\,\cdot)\|_{\dot{B}^\sigma_{p,r}}
\approx \lambda^{\sigma-\frac{d}{p}}\|f\|_{\dot{B}^\sigma_{p,r}} .
\]

\medskip
\noindent
$\bullet$ \emph{Completeness.}  
The space $\dot{B}^\sigma_{p,r}$ is Banach whenever $\sigma<\frac{d}{p}$, or $\sigma=\frac{d}{p}$ and $r=1$.

\medskip
\noindent
$\bullet$ \emph{Interpolation.}  
Let $\sigma_1\neq\sigma_2$, $\theta\in(0,1)$, and $1\le p,r_1,r_2,r\le\infty$ with
\[
\frac1r=\frac{\theta}{r_1}+\frac{1-\theta}{r_2}.
\]
Then
\[
\|f\|_{\dot{B}^{\theta\sigma_1+(1-\theta)\sigma_2}_{p,r}}
\lesssim 
\|f\|_{\dot{B}^{\sigma_1}_{p,r_1}}^\theta
\|f\|_{\dot{B}^{\sigma_2}_{p,r_2}}^{1-\theta}.
\]

\medskip
\noindent
$\bullet$ \emph{Fourier multipliers.}  
If $F$ is a smooth homogeneous function of degree $m$ on $\mathbb{R}^d\setminus\{0\}$, then
\[
F(D):\dot{B}^\sigma_{p,r}\longrightarrow \dot{B}^{\sigma-m}_{p,r}.
\]


The following embedding properties will be used repeatedly throughout the paper.

\begin{prop}\label{Prop2.1}
The following statements hold:
\begin{itemize}
\item For any $p\in[1,\infty]$, we have the continuous embeddings
\[
\dot{B}^{0}_{p,1}\hookrightarrow L^{p}\hookrightarrow \dot{B}^{0}_{p,\infty}.
\]

\item Let $\sigma\in\mathbb{R}$, $1\leq p_{1}\leq p_{2}\leq\infty$ and $1\leq r_{1}\leq r_{2}\leq\infty$.
Then
\[
\dot{B}^{\sigma}_{p_1,r_1}\hookrightarrow
\dot{B}^{\sigma-d\left(\frac{1}{p_{1}}-\frac{1}{p_{2}}\right)}_{p_{2},r_{2}}.
\]

\item The space $\dot{B}^{\frac{d}{p}}_{p,1}$ is continuously embedded into the space of bounded
continuous functions, which additionally vanish at infinity if $p<\infty$.
\end{itemize}
\end{prop}

\medskip

We also recall the classical \emph{Bernstein inequality}:
\begin{equation}\label{Eq:2.6}
\|D^{k}f\|_{L^{b}}
\leq C^{1+k} \lambda^{k+d\left(\frac{1}{a}-\frac{1}{b}\right)}\|f\|_{L^{a}},
\end{equation}
which holds for all functions $f$ such that $\mathrm{Supp}\,\widehat{f}\subset\{\xi\in\mathbb{R}^{d}:|\xi|\leq R\lambda\}$ 
for some $R>0$ and $\lambda>0$, provided that $k\in\mathbb{N}$ and $1\leq a\leq b\leq\infty$.

More generally, if $\mathrm{Supp}\,\widehat{f}\subset\{\xi\in\mathbb{R}^{d}:R_{1}\lambda\leq|\xi|\leq R_{2}\lambda\}$ for some $0<R_{1}<R_{2}$ and $\lambda>0$, then for any smooth homogeneous function $A$ of degree $m$
on $\mathbb{R}^{d}\setminus\{0\}$ and any $1\leq a\leq\infty$, one has (see e.g. Lemma~2.2 in \cite{BCD})
\begin{equation}\label{Eq:2.7}
\|A(D)f\|_{L^{a}}\approx \lambda^{m}\|f\|_{L^{a}}.
\end{equation}
As a direct consequence of \eqref{Eq:2.6} and \eqref{Eq:2.7}, we have
\[
\|D^{k}f\|_{\dot{B}^{s}_{p,r}}\approx \|f\|_{\dot{B}^{s+k}_{p,r}},
\qquad k\in\mathbb{N}.
\]

\medskip

\subsection{Pseudo-product law}
Product estimates in Besov spaces play a fundamental role in the control of nonlinear terms. Let us start with estimates for simpple product:
\begin{lem}\label{aaaa}
Let $s>0$ and $1< p<\infty$. Then $\dot{W}^{s,p}\cap L^{\infty}$ is an algebra and
\[
\|ab\|_{\dot{W}^{s,p}}
\lesssim
\|a\|_{L^{\infty}}\|b\|_{\dot{W}^{s,p}}
+
\|b\|_{L^{\infty}}\|a\|_{\dot{W}^{s,p}}.
\]
\end{lem}

Furthermore, in consideration of appearance for Fourier symbol $$\mathcal{A}^s:=\Lambda^{s}(1-\Delta)^{\frac{1}{2}};\quad
\mathcal{A}_1^s:=\Lambda^{s}U^{-1}_\varepsilon;\quad
\mathcal{A}_2^s:=\Lambda^{s}U^{-1}_\varepsilon(1-\Delta)^{-\frac{1}{2}},$$
we also need the following pseudo-product involving different symbols:
\begin{lem}\label{eeee}
Let $s>0$. Then the following inequalities hold true
\begin{equation}\label{pp1}
\|\mathcal{A}^s(ab)\|_{L^2}\leq C\big(
\|a\|_{L^{\infty}}\|b\|_{H^{s+1}}+\|b\|_{L^{\infty}}\|a\|_{H^{s+1}}\big);
\end{equation}
\begin{equation}\label{pp2}
\|\mathcal{A}^s_1(ab)\|_{L^2}\leq C\varepsilon^{-\frac 12}\big(
\|a\|_{L^{\infty}}\|b\|_{H^{s}}+\|b\|_{L^{\infty}}\|a\|_{H^{s}}\big);
\end{equation}
\begin{equation}\label{pp3}
\|\mathcal{A}^s_2(ab)\|_{L^2}\leq C\big(
\|a\|_{L^{\infty}}\|b\|_{H^{s}}+\|b\|_{L^{\infty}}\|a\|_{H^{s}}\big),
\end{equation}
where constant $C>0$ independent of $\varepsilon$.
\end{lem}

\begin{proof}
For \eqref{pp1}, there holds
\begin{equation}
\|\mathcal{A}^s(ab)\|_{L^2}\leq\|\Lambda^s(ab)\|_{L^2}+\|\Lambda^{s+1}(ab)\|_{L^2}\leq C\big(
\|a\|_{L^{\infty}}\|b\|_{H^{s+1}}+\|b\|_{L^{\infty}}\|a\|_{H^{s+1}}\big).
\end{equation}
For \eqref{pp2}, in light of the fact
$$|\U^{-1}_\varepsilon(|\xi|)|\lesssim\varepsilon^{-\frac 12},$$
the classical product law indicates
\begin{equation}
\|\mathcal{A}^s_1(ab)\|_{L^2}\leq\varepsilon^{-\frac 12}\|\Lambda^s(ab)\|_{L^2}\leq C\varepsilon^{-\frac 12}\big(
\|a\|_{L^{\infty}}\|b\|_{H^{s}}+\|b\|_{L^{\infty}}\|a\|_{H^{s}}\big).
\end{equation}
Finally, in terms of \eqref{pp3}, notice that $|\U^{-1}_\varepsilon(|\xi|)(1+|\xi|^2)^{-\frac 12}|$ is bounded, we have
\begin{equation}
\|\mathcal{A}^s_2(ab)\|_{L^2}\leq\|\Lambda^{s}(ab)\|_{L^2}\leq C\big(
\|a\|_{L^{\infty}}\|b\|_{H^{s}}+\|b\|_{L^{\infty}}\|a\|_{H^{s}}\big).
\end{equation}

\end{proof}

Furthermore, we have the following lemma for commutator estimates:
\begin{lem}\label{commutator}
Let $s>0$. Then the following inequalities hold true
\begin{equation}\label{qq1}
\|[\mathcal{A}^s,a\cdot\nabla]b\|_{L^2}\leq C\big(
\|\nabla a\|_{L^{\infty}}\|b\|_{H^{s+1}}+\|\nabla b\|_{L^{\infty}}\|a\|_{H^{s+1}}\big);
\end{equation}
\begin{equation}\label{qq2}
\|[\mathcal{A}^s_1,a\cdot\nabla]b\|_{L^2}\leq C\varepsilon^{-\frac 12}\big(
\|\nabla a\|_{L^{\infty}}\|b\|_{H^{s}}+\|\nabla b\|_{L^{\infty}}\|a\|_{H^{s}}\big);
\end{equation}
\begin{equation}\label{qqxi3}
\|[\mathcal{A}^s_2,a\cdot\nabla]b\|_{L^2}\leq C\big(
\|\nabla a\|_{L^{\infty}}\|b\|_{H^{s}}+\|\nabla b\|_{L^{\infty}}\|a\|_{H^{s}}\big),
\end{equation}
where constant $C>0$ independent of $\varepsilon$.
\end{lem}

\begin{proof}
The proof of the first commutator is standard by a Kato-Ponce's argument in \cite{KP0}, For the second one, there holds
$$[\mathcal{A}^s,a\cdot\nabla]b=\int_{\R^3\times \R^3}e^{ix\cdot\xi}\big(|\xi|^s\hat U^{-1}_\varepsilon(|\xi|)-|\eta|^s\hat U^{-1}_\varepsilon(|\eta|)\big)\hat a(\xi-\eta)\hat {\nabla b}(\eta)d\xi d\eta$$
Now by frequency localized function, we could decompose the frequency space as \eqref{Bony decom}. Now we only focus on the main case $|\xi-\eta|\ll|\xi|\sim|\eta|$ while other situations are easier. By the mean value theory, we have for $g(x)=x^s\hat U^{-1}_\varepsilon(x), x>0$ and $y=|\xi|+\theta(|\xi|-|\eta|), \theta\in(0,1)$ that
\begin{eqnarray*}
\big(|\xi|^s\hat U^{-1}_\varepsilon(|\xi|)-|\eta|^s\hat U^{-1}_\varepsilon(|\eta|)\big)&=&(|\xi|-|\eta|)g'(y)\\
&=&(|\xi|-|\eta|)\left(sy^{s-1}\hat U^{-1}_\varepsilon(y)-y^s\frac{\U'_\varepsilon(y)}{\hat U_\varepsilon(y)}\right).
\end{eqnarray*}
In light of \eqref{U-1} and the fact $||\xi|-|\eta||\leq|\xi-\eta|$, we get to
\begin{eqnarray*}
\left|\big(|\xi|^s\hat U^{-1}_\varepsilon(|\xi|)-|\eta|^s\hat U^{-1}_\varepsilon(|\eta|)\big)\right|
\lesssim\varepsilon^{-\frac{1}{2}}(|\xi-\eta|)|\xi|^{s-1}.
\end{eqnarray*}
Hence routine calculations leads to 
\begin{multline}
\left\|\int_{\R^3\times \R^3}e^{ix\cdot\xi}\big(|\xi|^s\hat U^{-1}_\varepsilon(|\xi|)-|\eta|^s\hat U^{-1}_\varepsilon(|\eta|)\big)\hat a(\xi-\eta)\hat {\nabla b}(\eta)d\xi d\eta\right\|_{L^2}\\
\lesssim\varepsilon^{-\frac{1}{2}}\|\nabla a\|_{L^\infty}\|\nabla b\|_{H^{s-1}}\lesssim\varepsilon^{-\frac{1}{2}}\|\nabla a\|_{L^\infty}\| b\|_{H^{s}}
\end{multline}
for $|\xi-\eta|\ll|\xi|\sim|\eta|$. Other situations are easier to handle and we obtain \eqref{qq2}. The last commutator estimate follows the same fashion as above, and we omit the detailed estiamtes.
\end{proof}

\bigbreak\bigbreak\noindent
\noindent 
\textbf{Acknowledgments.~} The authors sincerely thank Professor Kenji Nakanishi for helpful discussions and advice in the course of this research. 

\smallbreak
\noindent
\textbf{Conflict of interest.} The authors do not have any possible conflicts of interest.

\smallbreak
\noindent
\textbf{Data availability statement.}
 Data sharing is not applicable to this article, as no data sets were generated or analyzed during the current study.


\begin{thebibliography}{99}
\addcontentsline{toc}{section}{Reference}

\bibitem{AH}
C. Audiard and B. Haspot, Global well-posedness of the Euler-Korteweg system for small irrotational data, {\it Commun. Math. Phys.}, {\bf 351}, 201-247 (2017).

\bibitem{BCD}
H. Bahouri, J.-Y. Chemin and R. Danchin, {\it Fourier Analysis and Nonlinear Partial Differential Equations,} Grundlehren der mathematischen Wissenschaften, {\bf 343}, Springer (2011).

\bibitem{BDDN}
X. Blanc, R. Danchin, B. Ducomet and S. Ne$\check{c}$asov$\acute{a}$, The global existence issue for the compressible Euler system
with Poisson or Helmholtz couplings, {\it Journal of Hyperbolic Differential Equations,} {\bf 18}, 169-193 (2021).

\bibitem{CG}
S. Cordier and E. Grenier, Quasineutral limit of an Euler-Poisson system arising from plasma physics.
{\it Comm. Part. Diff. Eqs.}, {\bf 25}, 1099–1113 (2000).

\bibitem{CH}
T. Candy and S. Herr, The massless and the non-relativistic limit for the cubic Dirac equation, {\it arXiv:2308.12057}, (2023).

\bibitem{CCZ}
J. Carrillo, Y. Choi and E. Zatorska, 
On the pressureless damped Euler-Poisson equations with quadratic confinement: critical thresholds and large-time behavior. {\it Math. Models Methods Appl. Sci.}, {\bf 26},  2311-2340 (2016).

\bibitem{CSW}
F. Cavalletti, M. Sedjro and M. Westdickenberg, 
A simple proof of global existence for the 1D pressureless gas dynamics equations. {\it SIAM J. Math. Anal.}, {\bf 47}, 66-79 (2015).

\bibitem{CW}
T. Cazenave and F. Weissler, Rapidly decaying solutions of the nonlinear Schr$\mathrm{\ddot{o}}$dinger equation. {\it Commun. Math. Phys.}, {\bf 147}, 75-100 (1992).

\bibitem{CHWY}
G. Chen, L. He, Y. Wang and D. Yuan, 
Global solutions of the compressible Euler-Poisson equations with large initial data of spherical symmetry. {\it Commun. Pure Appl. Math.}, {\bf 77}, 2947-3025 (2024).

\bibitem{CHLWW}
G. Chen, F. Huang, T. Li, W. Wang and Y. Wang, 
Global finite-energy solutions of the compressible Euler-Poisson equations for general pressure laws with large initial data of spherical symmetry. {\it Commun. Math. Phys.}, {\bf 405},  85 pp. (2024).

\bibitem{CW2}
G. Chen and D. Wang,
Convergence of shock capturing schemes for the compressible Euler-Poisson equations. {\it Commun. Math. Phys.}, {\bf 179}, 333-364 (1996).

\bibitem{CJL}
Y. Choi, J. Jung and Y. Lee, Global smooth solutions to the irrotational Euler-Riesz system in three dimensions, {\it Trans. Amer. Math. Soc.}, {\bf 379}, 241-288 (2026).

\bibitem{CKKT}
Y. Choi, D. Kim, D.  Koo and E. Tadmor, 
Critical thresholds in pressureless Euler-Poisson equations with background states. 
{\it Ann. Inst. Henri Poincar$\mathrm{\acute{e}}$ C, Anal. Non Lin$\mathrm{\acute{e}}$aire}, {\bf 43},  203-237 (2026).

\bibitem{DIP}
Y. Deng; A.D. Ionescu and B. Pausader, The Euler-Maxwell system for electrons: global solutions in 2D, {\it Arch. Ration. Mech. Anal.}, {\bf 225}, 771-871 (2017).

\bibitem{DIPP}
Y. Deng; A.D. Ionescu; B. Pausader and F. Pusateri, Global solutions of the gravity-capillary water-wave system in three dimensions, {\it Acta Math.}, {\bf 219}, 213-402 (2017).

\bibitem{GM}
P. Germain and N. Masmoudi, Global existence for the Euler-Maxwell system, {\it Ann. Sci. Ec. Norm. Super.}, {\bf 47}, 469-503 (2014).

\bibitem{ELT}
S. Engelberg, H. Liu and E. Tadmor, 
Critical thresholds in Euler-Poisson equations. {\it Indiana Univ. Math. J.}, {\bf 50}, 109-157 (2001).

\bibitem{GMS1}
P. Germain, N. Masmoudi and J. Shatah, Global solutions for 2D quadratic Schr$\mathrm{\ddot{o}}$dinger equations, {\it J. Math. Pures Appl.}, {\bf 97}, 505-543 (2012).

\bibitem{GMS2}
P. Germain, N. Masmoudi and J. Shatah, Global solutions for 3D quadratic Schr$\mathrm{\ddot{o}}$dinger equations, {\it Int. Math. Res. Not.}, {\bf 3}, 414-432 (2009).

\bibitem{GMS3}
P. Germain, N. Masmoudi and J. Shatah, Global solutions for the gravity water waves equation in dimension 3, {\it Ann. of Math.}, {\bf 175}, 691-754 (2012).

\bibitem{G}
Y. Guo, Smooth irrotational flows in the large to the Euler-Poisson system in $\mathbb{R}^{3+1}$, {\it Comm. Math. Phys.}, {\bf 195}, 249-265 (1998).

\bibitem{GIP}
Y. Guo, A.D. Ionescu and B. Pausader, Global solutions of the Euler-Maxwell two-fluid system in 3D, Ann. of Math., 183, 377-498 (2016).

\bibitem{GP0}
Y. Guo and B. Pausader, Global smooth ion dynamics in the Euler-Poisson system, {\it Comm. Math. Phys.}, {\bf 303}, 89-125 (2011).

\bibitem{GP}
Y. Guo and X. Pu, KdV limit of the Euler-Poisson system, \textit{Arch. Ration. Mech. Anal.}, {\bf 211}, 673-710 (2014).

\bibitem{GN}
Z. Guo and K. Nakanishi, Small energy scattering for the Zakharov system with radial symmetry, {\it Int. Math. Res. Not.}, {\bf 9}, 2327-2342 (2014).

\bibitem{GPW}
Z. Guo; L. Peng and B. Wang, Decay estimates for a class of wave equations, {\it J. Funct. Anal.}, {\bf 6}, 1642-1660 (2008).

\bibitem{GNT1}
S. Gustafson, K. Nakanishi and T.-P. Tsai, Scattering for the Gross-Pitaevskii equation, \textit{Math. Res. Lett.}, {\bf 13}, 273-285 (2006).

\bibitem{GNT2}
S. Gustafson, K. Nakanishi and T.-P. Tsai, Scattering theory for the Gross-Pitaevskii equation in three dimensions, \textit{Commun. Contemp. Math.}, {\bf 11}, 657-707 (2009).

\bibitem{HJ}
M. Had$\mathrm{\check{z}}$i$\mathrm{\acute{c}}$ and J. Jang, A class of global solutions to the Euler-Poisson system. \textit{Comm. Math. Phys.}, {\bf 370}, 475-505 (2019).

\bibitem{IP}
A.D. Ionescu and B. Pausader, The Euler-Poisson system in 2D: global stability of the constant equilibrium solution, {\it Int. Math. Res. Not.}, {\bf 4}, 761-826 (2013).

\bibitem{IP2}
A.D. Ionescu and B. Pausader, Global solutions of quasilinear systems of Klein-Gordon equations in 3D, {\it J. Eur. Math. Soc.}, {\bf 16}, 2355-2431 (2014).

\bibitem{IP3}
A.D. Ionescu A D and F. Pusateri, Global solutions for the gravity water waves system in 2d, {\it Invent. math}, {\bf 199}, 653-804 (2015).

\bibitem{JLZ}
J. Jang; D. Li and X. Zhang, Smooth global solutions for the two-dimensional Euler Poisson system, {\it Forum Math.}, {\bf 26}, 645-701 (2014).

\bibitem{K}
T. Kato, The Cauchy problem for quasilinear symmetric systems. {\it Arch. Rat. Mech. Anal.}, {\bf 58}, 181–205 (1975).

\bibitem{KP0}
T. Kato and G. Ponce, Commutator estimates and the Euler and Navier-Stokes equations. {\it Commun. Pure Appl. Math.}, {\bf 41}, 891-907 (1988).

\bibitem{KP}
S. Klainerman and G. Ponce, Global, small amplitude solutions to nonlinear evolution equations, {\it Commun. Pure Appl. Math.}, {\bf 36}, 133-141 (1983).

\bibitem{LW}
D. Li and Y. Wu, The Cauchy problem for the two dimensional Euler-Poisson system, {\it J. Eur. Math. Soc.}, {\bf 16}, 2211-2266 (2014).

\bibitem{LT1}
H. Liu and E. Tadmor, Critical thresholds in 2D restricted Euler-Poisson equations. {\it SIAM J. Appl. Math.}, {\bf 63}, 1889–1910 (2003).

\bibitem{LT2}
H. Liu and E. Tadmor, Spectral dynamics of the velocity gradient field in restricted flows. {\it Commun Math.
Phys.}, {\bf 228}, 435–466 (2002).

\bibitem{LRXX}
T. Luo, J. Rauch, C. Xie and Z. Xin, 
Stability of transonic shock solutions for one-dimensional Euler-Poisson equations. {\it Arch. Ration. Mech. Anal.}, {\bf 202}, 787-827 (2011).

\bibitem{MN}
N. Masmoudi and K. Nakanishi, Multifrequency NLS scaling for a model equation of gravity-capillary waves. {\it Commun. Pure Appl. Math.}, {\bf 66}, 1202-1240 (2013).

\bibitem{NT}
T. Nguyen and A. Tudorascu,
One-dimensional pressureless gas systems with/without viscosity. {\it Commun. Partial Differ. Equations}, {\bf 40}, 1619-1665 (2015).

\bibitem{PW}
Y. Peng amd Y. Wang, Boundary layers and quasi-neutral limit in steady state Euler-Poisson equations
for potential flows. {\it Nonlinearity}, {\bf 17}, 835–849 (2004).

\bibitem{PL}
Y. Peng and C. Liu, 
Global quasi-neutral limit for a two-fluid Euler-Poisson system in several space dimensions. {\it SIAM J. Math. Anal.}, {\bf 55}, 1405-1438 (2023).

\bibitem{PL2}
Y. Peng and C. Liu, 
Global quasi-neutral limit for a two-fluid Euler-Poisson system in one space dimension. {\it J. Differ. Equations}, {\bf 330}, 81-109 (2022).

\bibitem{Sh}
J. Shatah, Global existence of small solutions to nonlinear evolution equations, {\it J. Differential Equations}, {\bf 46}, 409-425 (1982).

\bibitem{Sideris}
T. Sideris, Formation of singularities in three-dimensional compressible fluids, {\it Commun. Math. Phys.}, {\bf 101}, 475-485, (1985).

\bibitem{Song}
Z. Song, Global dynamics and vanishing capillarity limit of the 3D Euler-Poisson-Korteweg equations. {\it Math. Ann.}, {\bf 392}, 2157-2224 (2025).

\bibitem{St}
W. Strauss, Nonlinear scattering theory at low energy, {\it J. Funct. Anal.}, {\bf 41}, 110-133 (1981).

\bibitem{TW}
E. Tadmor and D. Wei, On the global regularity of subcritical Euler–Poisson equations with pressure, {\it J. Eur. Math. Soc.}, {\bf 10}, 757-769 (2008).

\bibitem{W}
S. Wang, Quasineutral limit of Euler-Poisson system with and without viscosity, {\it Discrete Contin. Dyn. Syst.}, {\bf 29}, 419-456 (2005).

\bibitem{W2}
X. Wang, Global solution for the 3D quadratic Schr$\mathrm{\ddot{o}}$dinger equation of $Q(u,\bar u)$ type, {\it Discrete Contin. Dyn. Syst.}, {\bf 37}, 5037-5048 (2017).

\bibitem{Xiao}
J. Xiao, Global weak entropy solutions to the Euler-Poisson system with spherical symmetry. {\it Math. Models Methods Appl. Sci.}, {\bf 26}, 1689-1734 (2016).

\end{thebibliography}
\end{document}